\documentclass[twoside,11pt]{article}
\usepackage{graphicx}
\usepackage{subcaption}
\usepackage{blindtext}
\usepackage{jmlr2e}
\usepackage{amsmath, amssymb, amsfonts}
\usepackage{bm}
\usepackage{algorithm}
\usepackage{algpseudocode}%
\usepackage{mathrsfs}
\usepackage{mathtools}
\usepackage{enumitem}
\usepackage{multirow}
\usepackage{booktabs}
\usepackage{tikz}
\usepackage{tikz-cd}
\usepackage{float}
\usepackage{lastpage}
\usepackage{placeins}
\usepackage{placeins}   
\usepackage{booktabs}   
\usepackage{multirow}   
\usepackage{siunitx}    
\usetikzlibrary{arrows.meta, positioning}

\newcommand{\R}{\mathbb{R}}
\newcommand{\E}{\mathbb{E}}
\newcommand{\N}{\mathcal{N}}

\newcommand{\Pm}{P_m}
\newcommand{\Pmp }{P_m^\perp}
\newcommand{\Qm}{Q_m}
\newcommand{\Tm}{T_m}

\newcommand{\norm}[1]{\left\|#1\right\|}

\newcommand{\tr}{\mathrm{tr}}

\firstpageno{1}

\begin{document}

\title{Stochastic Lanczos Quadrature for Computational Uncertainty in Linear Algebra}

\author{%
\name Hassan Fifen \email hfifen@aimsric.org \\
\addr Department of Mathematics and Statistics\\
University of Dodoma, Dodoma, Tanzania\\
African Institute for Mathematical Sciences Research and Innovation Centre\\
Kicukiro, Kigali, Rwanda
\AND
\name Issa Karambal \email ikarambal@aimsric.org \\
\addr African Institute for Mathematical Sciences Research and Innovation Centre\\
Kicukiro, Kigali, Rwanda
\AND
\name Philipp Hennig \email philipp.hennig@uni-tuebingen.de \\
\addr Department of Computer Science\\
University of T\"ubingen, T\"ubingen, Germany\\
T\"ubingen AI Center, T\"ubingen, Germany
}

\editor{}

\maketitle

\begin{abstract}
Applying a matrix function to a vector is a common operation in large-scale scientific computing, such as  Bayesian inference. When the matrix is symmetric positive definite and the function is its inverse, the Lanczos algorithm gives a matrix-free approximation from matrix-vector products alone. Stopping after a fixed number of steps keeps the approximation on the explored Krylov subspace but assigns nothing to its orthogonal complement, which holds the smallest eigenvalues and therefore the largest contributions to the inverse. This truncation is usually treated as a silent numerical error rather than as measurable uncertainty. We show that a single Lanczos run yields both a matrix-free approximation of the inverse and a Gaussian distribution over the exact result, whose covariance measures what the truncation discarded and concentrates as the number of steps grows. The covariance restores the complement in two parts: a rank-one coupling at the truncation boundary, given exactly by the Lanczos residual, and an isotropic term on the remaining bulk, whose size we estimate by Projected Stochastic Lanczos Quadrature, reusing one Krylov basis across all probes. We prove consistency, with bias decaying exponentially in the quadrature depth and variance at the Monte Carlo rate.
\end{abstract}

\begin{keywords}
Krylov subspace methods, Lanczos algorithm, Stochastic trace estimation, Gauss quadrature, Computational uncertainty.
\end{keywords}

\section{Introduction}

Many problems in machine learning and scientific computing rely on a single operation, such as applying the inverse of a large symmetric positive-definite matrix to a vector, or using it as the covariance of a Gaussian distribution.  This is the case for posterior sampling in Bayesian neural networks \citep{mackay1992practical, daxberger2021laplace}, for sampling from Gaussian process priors \citep{rasmussen2006}, and for whitening in second-order optimization \citep{martens2015optimizing}. In large-scale models, operations such as inverting or factoring the corresponding matrix are extremely expensive. A  free matrix-vector product algorithm can be a solution for this problem and  Lanczos algorithm \citep{lanczos1950, saad2003iterative} is an example of such  free matrix-vector product. 

For a symmetric positive definite matrix $A$, starting from a given vector $b$ and running for $m$ steps, it builds an
orthonormal basis for the Krylov subspace spanned by $b, Ab, \ldots, A^{m-1}b$, and approximates an analytic function on $A$ inside that subspace using only matrix-vector products. In Bayesian deep learning this makes posterior sampling under the Laplace approximation
\citep{daxberger2021laplace} scalable, and it avoids the structural assumptions of diagonal \citep{mackay1992practical}, low-rank \citep{ritter2018}, or Kronecker-factored  \citep{martens2015optimizing} approximations, all of which lose  part of the curvature and leave the posterior underestimated in the directions where they are neglected.

But the Lanczos approach, used as a low-rank approximation, has a limitation that has not yet been  addressed.
Indeed, when we stop the iteration  after $m$ steps, the approximation lives entirely inside the resolved Krylov subspace and assigns exactly nothing to its orthogonal complement. That complement is $(d-m)$-dimensional, and it can retain  a large portion  of the inverse's spectral mass, because it contains the smallest eigenvalues of $A$, which become the largest under inversion. This problem is very critical  when $A$ is ill-conditioned, as deep network Hessians  are \citep{ghorbani2019investigation, papyan2020traces, sagun2017empirical}. The samples we draw are then too  concentrated; the sampler is confident about directions it has  never explored. In practice, this truncation error is simply ignored. It is not measured, and it is not taken into account  in the calculations. 

Probabilistic numerics \citep{hennig2022probabilistic, cockayne2019bayesian} gives us the right way to think about this. Instead of treating the result of a truncated iteration as a point with some unknown error, we treat the part of the computation we did not finish as a source of uncertainty, and we describe
it as a probability distribution. This idea has been developed for linear solvers in BayesCG \citep{cockayne2019bayesiancg} and for Gaussian process inference in computation-aware Gaussian processes \citep{wenger2022posterior}, both of which quantify the computational uncertainty of the linear solve $A^{-1}b$. We build on this  work and extend it to the matrix function $ A^{-1} $. The matrix function is a different case  because it cannot be written as a finite sum of solver iterates, and its complement carries a coupling at the truncation boundary that does not appear in the solve. The closest related correction is that of \citet{filippone2015enabling}, who use an unbiased Russian-roulette scheme for individual solves, and its variance grows as the stopping threshold is lowered.

We address this limitation directly by showing that the truncation error of the $m$-step Lanczos approximation to $A^{-1}$ lives in the orthogonal complement of the Krylov subspace, and that this missing uncertainty splits  into two parts. When we write $A$ in the Krylov basis, the resolved and unresolved subspaces  interact only through a single rank-one term, the Lanczos residual $\beta_m$. Because of this, the true complement block of the inverse is not the compressed operator $T_{m,\perp}^{-1}$, but its rank-one Schur correction $S_\perp^{-1}$, the compressed operator is just the special case where the coupling is ignored ($\beta_m = 0$). We keep this boundary coupling exactly, at matrix-free cost, and we approximate the remaining coupling-free bulk of the complement with a single isotropic term. To find the magnitude of that term, we introduce Projected Stochastic Lanczos Quadrature (P-SLQ), which performs short Lanczos runs from random probe vectors re-projected into the complement and applies Gauss quadrature to the resulting tridiagonal matrices. From a single Lanczos run, we then get two things at once: a matrix-free approximation of $A^{-1}$, and a Gaussian belief over the true $A^{-1}b$ whose covariance separates the variance the computation has resolved from the variance still left in the complement, and which shrinks to the exact answer as the resolved subspace grows.

\paragraph{Contributions.} Our contributions are the following.

\begin{enumerate}
\item \textbf{Computational uncertainty for the inverse.}
We show that the variance ignored by the $m$-step Lanczos approximation to $A^{-1}b$ lives entirely in the unresolved subspace, and we model it in the probabilistic-numerics sense with respect to  the true complement block of $A^{-1}$, rather than its truncated version. Writing $A$ in the Krylov basis shows that the resolved and unresolved subspaces meet only on the rank-one Lanczos residual $\beta_m$ (Lemma~\ref{lem:rankone_coupling}), so the true complement block is the Schur complement $S_{\perp}^{-1}$, with $S_{\perp}=T_{m,\perp}-\gamma\,e_1e_1^{\top}$ and $\gamma=\beta_m^2(T_m^{-1})_{mm}$, the compressed operator $T_{m,\perp}^{-1}$ is the decoupled case $\beta_m=0$. The corrected covariance
\[
\Sigma_m
=Q_mT_m^{-1}Q_m^{\top}
+\bigl(S_{\perp}^{-1}\bigr)_{11}q_{\parallel}q_{\parallel}^{\top}-q_{\parallel}q_{\perp}^{\top}-q_{\perp}q_{\parallel}^{\top}
+\sigma_u uu^{\top}
+\omega(m)\,P_{\mathrm{rest}},
\]
 where $\omega(m)=\frac{\mathrm{tr}(S_{\perp}^{-1})-\sigma_u}{d-m-1}$ keeps the coupling between the two subspaces exactly and approximates isotopically  only the coupling-free bulk, whose error is controlled by the complement condition number $\kappa_S^{\perp}(m)$ (Corollary~\ref{cor:frobenius}, Theorem~\ref{thm:iso},  Proposition~\ref{prop:kappa_sperp}). The standard low-rank Krylov approximation gives the whole complement zero variance and drops the coupling, we restore both, matrix-free. The same corrected covariance serves for two roles: $\Sigma_m$ approximates $A^{-1}$, and its complement part $\Sigma_m^{\perp}$ is the covariance
of the computational-uncertainty distribution $\mathcal{N}(f_m(A)b,\Sigma_m^{\perp})$ over the true $A^{-1}b$, which contracts to a point as $m\to d$. Here $f_m$ is the matrix-function approximating $A^{-1}$

\item \textbf{Projected Stochastic Lanczos Quadrature (P-SLQ).}
We give a matrix-free algorithm that estimates $\omega(m)$ using only
matrix-vector products with $A$. P-SLQ draws $N$ random probes, projects them onto the complement, runs $l$-step re-projected Lanczos processes there, and applies Gauss quadrature to the resulting tridiagonal matrices. Because every probe reuses the single Krylov basis, the total cost is $O(m+Nl)$ products (with $m, l < N$), against $O(Nm)$ for methods that rebuild a Krylov space per sample. We prove the estimator is consistent, with bias decaying as $O(\rho_{\perp}^{-2l})$ in the quadrature depth and variance at the optimal Monte Carlo rate $O(N^{-1})$ (Theorem~\ref{thm:pslq_consistency}), since the quadrature runs on the compressed complement operator, its rate is set by $\kappa^{\perp}(m)\le\kappa$ and is at least as fast as the global Lanczos rate. The coupling-corrected magnitude then follows from a single boundary probe by a deterministic Sherman-Morrison step (Section~\ref{sec:sperp_estimation}).

\item \textbf{Application: Laplace approximation in Bayesian deep learning.}
We test the corrected covariance on multiple scales  on the Laplace approximation in Bayesian deep learning \citep{daxberger2021laplace}, on FashionMNIST/CNN ($p{=}11{,}878$) where the exact full Laplace is computable,  and on CIFAR‑10/ResNet‑9 ($p{\approx}673,000$) for scalability, $\Sigma_m$ achieves near‑zero per‑point KL divergence to the exact posterior and matches the exact NLL and ECE within standard errors, while requiring only matrix‑vector products.  Additionally, on trace estimation benchmarks, P‑SLQ attains lower relative error than ULISSE and some benchmark methods at matched matrix‑vector budgets across three conditioning regimes.
\end{enumerate}

The remainder of this paper is organized as follows. In the next subsections, Section~\ref{sec:relatedwork} reviews related work on Krylov methods, stochastic trace estimation, and probabilistic numeric. Section~\ref{sec:method} formulates the inverse
approximation problem and recalls the Lanczos error bound.
Section~\ref{sec:cu} derives the complement covariance and
justifies the isotropic approximation. Section~\ref{sec:pslq} introduces P-SLQ and proves its statistical guarantees. Section~\ref{sec:experiments} presents experiments for validating our approach and Section~\ref{sec:conclusion} concludes.

\section{Related Work}
\label{sec:relatedwork}
Our work intersects four main research directions: Krylov methods for matrix functions, stochastic trace estimation, Laplace approximation for Bayesian deep learning and probabilistic numeric. We review each of these lines of work below and highlight where our contribution diverges from, and complements, the closest existing methods.

\subsection{Krylov Methods for Matrix Functions}
\label{sec:rw_krylov}
Approximating $f(A)b$ where $f$ is an analytic function   by Krylov subspace projection is classic in numerical linear algebra \citep{higham2008functions, saad2003iterative, golub2009matrices,liesen2013krylov}. The Lanczos algorithm \citep{lanczos1950} uses $m$ matrix-vector products to build an orthonormal basis $Q_m$ of $\mathcal{K}_m(A, b) = \mathrm{span}\{b, Ab, \ldots, A^{m-1}b\}$ and a tridiagonal matrix $T_m$ satisfying $AQ_m = Q_mT_m + \beta_m q_{m+1}e_m^\top$, yielding $f_m(A)b = \|b\|\, Q_m f(T_m) e_1$. The error is measured  by how well a degree-$m$
polynomial matches $f$ on the spectrum of $A$, for $f(t)=t^{-1}$,
a standard Chebyshev  \citep{saad1992analysis, druskin1989} give geometric decay in $m$ at a rate set by the condition number of $A$. \citet{chen2022} extended such bounds to a large classes of smooth functions, \citet{amsel2024nearly} showed Lanczos is nearly optimal among Krylov methods for $t^{-1}$, and \citet{meurant2006} showed that loss of orthogonality in finite precision does not prevent convergence. These results all concern the vector error $\|A^{-1}b - f_m(A)b\|_2$  and form the backbone of our analysis.

What this literature does not address is a  consequence coming from truncation: $f_m(A)$ has rank $m$ and ignores what happens to the remaining $(d-m)$-dimensional complement by  assigning  exactly zero on it,  accurate the polynomial approximation is. When $A$ is ill-conditioned and $m$ moderate, as for deep network Hessians \citep{ghorbani2019investigation, papyan2020traces, sagun2017empirical} this complement carries a large portion of the inverse spectral mass. The rank-one coupling between the resolved subspace and its complement through the residual $\beta_m$ is itself classical \citep{golub2009matrices, saad2003iterative}, we use it not to bound the vector error but to identify the true complement block of the inverse and to keep the boundary coupling inside a computational uncertainty covariance, which prior work does not do.

\subsection{The Laplace Approximation in Deep Learning}
\label{sec:rw_laplace}
The evaluation of $A^{-1}$ arises  as a Laplace
posterior \citep{mackay1992practical}, with $A$ the damped Gauss-Newton Hessian at the Maximum a posteriori estimate. \citet{daxberger2021laplace} unify several variants
(diagonal, KFAC \citep{martens2015optimizing}, low-rank, subnetwork) in a single framework. Alternative approaches such as MC Dropout \citep{gal2016dropout}, Deep Ensembles \citep{lakshminarayanan2017simple}, SWAG \citep{maddox2019simple}, and
sampling methods such as HMC \citep{neal2012bayesian} and its stochastic-gradient variants \citep{welling2011bayesian, chen2014stochastic} trade off guarantees or cost \citep{izmailov2021bayesian}. The common Hessian simplifications present  one flaw: they save computation by discarding curvature and assigning exactly zero posterior variance to the discarded directions, leaving the posterior overconfident and underestimated. The matrix-free Lanczos approach avoids structural assumptions on $A$, but inherits the truncation problem of Section~\ref{sec:rw_krylov}. Our
method measures and corrects the uncertainty of the discarded subspace without any structural assumption on $A$.

\subsection{Stochastic Trace Estimation}
\label{sec:rw_trace}
 In this work, our magnitude $\omega(m)$ is estimated  using  the complement trace $\tr\bigl((\Pmp A \Pmp)^{+}\bigr)=\tr(T_{m,\perp}^{-1})$ by the Sherman-Morrison correction of Section~\ref{sec:sperp_estimation}, so we require a trace estimate within the Krylov complement. Hutchinson's estimator \citep{hutchinson1989stochastic} gives $\mathbb{E}[\xi^\top M \xi]=\tr(M)$ for a fixed symmetric $M$, Stochastic Lanczos Quadrature (SLQ) \citep{ubaru2017fast} combines it with Gauss quadrature on the Lanczos tridiagonal, with bias vanishing in the quadrature depth and variance $O(N^{-1})$ in the number of probes. Hutch++ \citep{meyer2021hutch} and XTrace \citep{persson2022xtrace}  reduce variance by removing or reusing dominant directions.

These methods assume a matrix independent of the probes. In our setting the complement projector $\Pmp$ is built from a reference vector in the first Lanczos run, before the probes are drawn, so the compressed operator $\Pmp A \Pmp$ is fixed for those probes and Hutchinson's identity applies directly. A naive application of SLQ  fails, because the Lanczos iterations can fall outside of $\mathrm{range}(\Pmp)$ unless they are re-projected at every step, corrupting the estimated spectral density. Our P-SLQ algorithm keeps the process inside the complement by re-projecting at each step, and we prove this construction is consistent.

\subsection{Probabilistic Numerics and Computation-Aware Inference}
\label{sec:rw_pn}
Probabilistic numerics \citep{hennig2022probabilistic, cockayne2019bayesian}
give the guiding principle: early stopping of an iterative algorithm creates uncertainty that should be modeled probabilistically and propagated downstream. \citet{cockayne2019bayesiancg} introduced BayesCG, a Bayesian conjugate gradient solver whose posterior contracts at rate $\rho^{-2m}$ but assigns zero variance to
the explored subspace with no correction for the rest.
\citet{wenger2022posterior} introduced Computation-Aware Gaussian Processes (CaGP), which quantify the computational uncertainty of a Gaussian process linear solve, and \citet{stankewitz2024contraction} derived contraction rates for Lanczos posteriors in GP regression. CaGP is the closest prior work to ours, and we build directly on its computation-aware viewpoint. It targets the linear solve $K^{-1}y$, which is written as a finite sum of solver iterates, and it corrects along the solver residual, whereas we model the full $(d-m)$-dimensional complement together with the exact rank-one coupling
at the truncation boundary. In short, we extend the computation-aware program from linear solves to matrix functions.

Another    alternative  which is very closed is the unbiased solver ULISSE \citep{filippone2015enabling}, which continues the conjugate gradient iteration for a random number of extra steps and reweighs the increments by their inverse survival probability (a Russian-roulette scheme), making each solve unbiased in expectation. ULISSE and our method share the recognition that the truncation tail should not simply be discarded, but they make opposite tradeoffs.
Unbiasedness costs variance: as truncation is pushed earlier the reweighting factors grow, in the worst case without bound. We instead form a bounded, deterministic model of the complement, an exact rank-one coupling plus an isotropic bulk whose deviation from the true inverse is controlled by the single computable diagnostic $\kappa_S^{\perp}(m)$ (Theorem~\ref{thm:iso}), with a stochastic component whose variance is $O(N^{-1})$ independent of $m$. Moreover,
the roulette telescoping applies only to the solve $A^{-1}b$ and it debiases a point estimate rather than restoring a covariance. ULISSE and P-SLQ thus occupy opposite ends of the bias-variance frontier for Krylov truncation, on different computational targets.

\subsection{Low-Rank-Plus-Shift Approximations}
\label{sec:rw_lowrankshift}

The idea behind the approximation by low-rank plus a simple diagonal or  a scalar shift is well known, and our isotropic bulk correction is well related. In Gaussian process regression, the FITC approximation
\citep{snelson2006sparse}  replaces the kernel matrix $K$ by
$K \approx Q + \mathrm{diag}(K-Q)$ with $Q$ a low-rank Nystr\"om term, exactly preserving the marginal variances. Restricting the correction to a single scalar, $K \approx Q + \tau I$, the natural trace-matching choice $\tau = \tfrac{1}{n}\tr(K-Q)$ equals the mean of the discarded eigenvalues, a heuristic used to stabilize Nystr\"om approximations and avoid overconfidence in sparse GP models \citep{williams2000using, bauer2016understanding}. The same constructions appear in deflation preconditioning for positive definite symmetric systems, where a low-rank coarse correction is combined with a shift fixed to the mean eigenvalue of the deflated matrix \citep{saad2003iterative,
tang2009comparison}. In all these cases a single scalar represents the unmodeled tail of the spectrum, preserving the total trace. Our bulk term  $\omega(m)$ is exactly such a trace-matching scalar, but with two differences: it is applied to the coupling-free part of the Krylov complement after the exact boundary coupling has been removed, and it is estimated matrix-free by P-SLQ rather than assumed known.

\section{Matrix-Free Approximation via Lanczos}
\label{sec:method}

\subsection{Problem Statement}
\label{sec:problem}
Let $A \in \mathbb{R}^{d \times d}$ be symmetric positive-definite with
eigen decomposition $A = U \Lambda U^\top$,
$\Lambda = \mathrm{diag}(\lambda_1, \ldots, \lambda_d)$,
$\lambda_1 \ge \cdots \ge \lambda_d > 0$, condition number
$\kappa = \frac{\lambda_1}{\lambda_d}$, and $U$ orthogonal.
Our object of interest here  is the inverse $A^{-1}$, which we access only by matrix-vector products $v \mapsto Av$ as available, for instance, by automatic differentiation when $A$ is a Hessian matrix, without having to  form or store $A$.

The inverse $A^{-1}$ serves as the covariance of the Gaussian
$\mathcal{N}(\theta_{\star}, A^{-1})$, and drawing samples from this distribution is a core primitive in Bayesian inference. A sample is obtained from $z \sim \mathcal{N}(0, I_d)$ via any factor $L$ with $LL^\top = A^{-1}$,
\begin{equation}
\label{eq:gaussian_sample}
\theta = \theta_{\star} + L\,z .
\end{equation}\

The inverse  square root $A^{-\frac{1}{2}}$ is one of those factor, but forming it is expensive because it cost $O(d^3)$ in memory, for $d$ reaching millions it becomes intractable, even matrix-free method as Lanczos for $A^{\frac{1}{2}}$ approximation has the same problem due to the truncation as described below.  We instead construct an explicit matrix-free factor $L$  from a single Lanczos run (Section~\ref{sec:cu}), with $  A^{-1} = LL^\top \approx \Sigma_m $, and never compute a matrix square root.

The Lanczos algorithm builds a rank-$m$ approximation of $A^{-1}$ from the Krylov subspace $\mathcal{K}_m(A, v_{\mathrm{ref}})$, generated from a reference vector $v_{\mathrm{ref}}$. By construction it acts only inside the resolved subspace $\mathrm{range}(P_m)$ and assigns exactly zero to the $(d-m)$-dimensional complement $\mathrm{range}(P_m^{\perp})$, even though that complement carries a  important fraction of the inverse spectral mass when $A$ is ill-conditioned. Our objective is not the rank-$m$ approximation itself, which is well known, but to quantify the computational uncertainty this truncation discards, the epistemic uncertainty from stopping the iteration at step $m$ and to restore it, using only matrix-vector products with $A$. We represent this uncertainty by a correction supported on the unresolved subspace with two components: a rank-one coupling at the truncation boundary, captured exactly by the Lanczos residual $\beta_m$, and an isotropic variance $\omega(m)$ on the uncoupled bulk. The remainder of the paper develops this correction (Section~\ref{sec:cu}) and a matrix-free estimator for its
magnitude (Section~\ref{sec:pslq}).

\subsection{The Lanczos Algorithm}
Given a  reference vector $q_1 = \frac{v_{\mathrm{ref}}}{ \|v_{\mathrm{ref}}\|_2} $ and an integer $m \geq 1$, the Lanczos algorithm \citep{lanczos1950, saad2003iterative} builds an orthonormal basis of the Krylov subspace
\begin{equation}
  \label{eq:krylov}
  \mathcal{K}_m(A, v_{\mathrm{ref}})
  = \mathrm{span}\{v_{\mathrm{ref}},\, Av_{\mathrm{ref}},\,
    A^2 v_{\mathrm{ref}},\, \ldots,\, A^{m-1} v_{\mathrm{ref}}\}
\end{equation}
via the three terms recurrence
\begin{equation}
  \label{eq:recurrence}
  A q_j
  = \alpha_j q_j + \beta_j q_{j+1} + \beta_{j-1} q_{j-1},
  \qquad j = 1, \ldots, m-1,
\end{equation}
where $\alpha_j = q_j^\top A q_j$ and
$\beta_j = \|Aq_j - \alpha_j q_j - \beta_{j-1}q_{j-1}\|_2$.
In matrix form,
\begin{equation}
  \label{eq:lanczos_relation}
  A Q_m = Q_m T_m + \beta_m q_{m+1} e_m^\top,
\end{equation}
where $Q_m = [q_1, \ldots, q_m] \in \mathbb{R}^{d \times m}$ has orthonormal columns, $Q_m^\top Q_m = I_m$, and $T_m \in \mathbb{R}^{m \times m}$ is the symmetric tridiagonal matrix
\begin{equation}
  \label{eq:tridiagonal}
  T_m = Q_m^\top A Q_m
  =
  \begin{pmatrix}
    \alpha_1 & \beta_1  &          &          \\
    \beta_1  & \alpha_2 & \ddots   &          \\
             & \ddots   & \ddots   & \beta_{m-1} \\
             &          & \beta_{m-1} & \alpha_m
  \end{pmatrix}.
\end{equation}
The matrix $T_m$ is the restriction of $A$ to $\mathcal{K}_m(A, v_{\mathrm{ref}})$, and its eigenvalues $\lambda_1^{(m)} \geq \cdots \geq \lambda_m^{(m)} > 0$ are the Ritz values of $A$ with respect to $\mathcal{K}_m(A, v_{\mathrm{ref}})$. The basis $Q_m$ is built once from $v_{\mathrm{ref}}$ and reused throughout: for the covariance approximation, for sampling, and, in Section~\ref{sec:pslq}, across all
probe vectors.

\subsection{Lanczos-Based Approximation of the Inverse}

The Lanczos approximation of the action of $A^{-1}$ is obtained by projecting the inverse onto the Krylov subspace. With $T_m = S \Lambda_m S^\top$ the eigendecomposition of $T_m$, the $m$-step approximation of $A^{-1}$ is
\begin{equation}
  \label{eq:lanczos_inverse}
  f_m(A) = Q_m T_m^{-1} Q_m^\top
  = Q_m S \Lambda_m^{-1} S^\top Q_m^\top ,
\end{equation}
computed from exactly $m$ matrix-vector products with $A$, at cost $O(md)$ in memory and $O(m^2 d)$ in time. Applied to a vector $b$, it yields the standard Krylov estimate $Q_m T_m^{-1} Q_m^\top b$ of the solve $A^{-1}b$ \citep{saad2003iterative, higham2008functions}.

The orthogonal projectors onto the resolved Krylov subspace and its complement are
\begin{equation}
  \label{eq:projectors}
  P_m = Q_m Q_m^\top,
  \qquad
  P_m^\perp = I_d - Q_m Q_m^\top .
\end{equation}
The approximation \eqref{eq:lanczos_inverse} acts entirely within
$\mathrm{range}(P_m)$ and assigns exactly zero to $\mathrm{range}(P_m^{\perp})$. Our focus is not on eliminating this rank-$m$ deficiency, but rather on analyzing and quantifying it. The computational uncertainty we address is precisely the variance discarded by this rank-m truncation, we approximate the missing covariance and restore it in the next section. In Section~\ref{sec:cu} we characterize the part  $f_m(A)$ discards on the complement, restore it exactly at the truncation boundary and isotropically on the bulk, and assemble the explicit matrix-free factor $[\,M_1\ M_2\ M_3\,]$ used for sampling in \eqref{eq:gaussian_sample}.

\section{Computational Uncertainty on the Krylov Complement}
\label{sec:cu}

Let $Q_m\in\mathbb{R}^{d\times m}$ be the Lanczos basis after $m$ iterations, and  $Q_{m,\perp} \in\mathbb{R}^{d\times(d-m)}$ be any orthonormal basis of the  orthogonal
complement of $Q_m$, so that $[\,Q_m\ \ Q_{m,\perp}\,]$ is orthogonal. We never form $Q_{m,\perp}$, in fact  it appears only in the
analysis, and all quantities below depend on it  through
$\mathrm{range}(\Pmp)$, hence are invariant to its choice. The
complement is used only through the projector $\Pmp = I_d - \Qm\Qm^{\top}$, applied as $v \mapsto v - \Qm(\Qm^{\top} v)$. Define the projectors
\[
P_m=Q_mQ_m^{\top},\qquad P_m^{\perp}=I_d-P_m=Q_{m,\perp} Q_{m,\perp}^{\top}.
\]
Writing $A$ in this basis gives the symmetric block form
\begin{equation}
\begin{bmatrix}Q_m^{\top}\\[0.5em] Q_{m,\perp}^{\top}\end{bmatrix} A
\begin{bmatrix}Q_m & Q_{m,\perp}\end{bmatrix}
=\begin{bmatrix}
Q_m^{\top}A Q_{m} & Q_m^{\top}A Q_{m,\perp} \\[1em] 
Q_{m,\perp}^{\top}A Q_{m} & Q_{m,\perp}^{\top}A Q_{m,\perp}
\end{bmatrix},
\label{eq:blocks}
\end{equation}
Denote $T_m = Q_m^{\top}A Q_{m}$ and $T_{m,\perp}=Q_{m,\perp}^{\top}A Q_{m,\perp}$. The block $T_{m,\perp}=Q_{m,\perp}^{\top}A Q_{m,\perp}$ is the compression of $A$ to the unresolved subspace, which  is  also the restriction to $\mathrm{range}(P_m^{\perp})$ of the compressed operator
\begin{equation}
A_{m,\perp}=P_m^{\perp} A\,P_m^{\perp}.
\label{eq:Bdef}
\end{equation}
Since $A$ is symmetric positive definite, so  $T_{m,\perp}$  is too , and $A_{m,\perp}$ is symmetric positive semi-definite, with Moore-Penrose pseudo-inverse
\begin{equation}
A_{m,\perp}^{+}=Q_{m,\perp} T_{m,\perp}^{-1}Q_{m,\perp}^{\top},
\label{eq:Bpinv}
\end{equation}
that is verified directly from the four Penrose conditions using
$Q_{m,\perp}^{\top}Q_{m,\perp}=I_{d-m}$, indeed $A_{m,\perp} A_{m,\perp}^{+}=A_{m,\perp}^{+} A_{m,\perp}=P_m^{\perp}$, so $A_{m,\perp}^{+}$
inverts $A_{m,\perp}$ on $\mathrm{range}(P_m^{\perp})$. That implies
\begin{equation}
\mathrm{tr}\big(A_{m,\perp}^{+}\big)
=\mathrm{tr}\big(T_{m,\perp}^{-1}\big),
\label{eq:trace-id}
\end{equation}
in fact, 
\begin{align*}
\mathrm{tr}\big(A_{m,\perp}^{+}\big)
&=\mathrm{tr}\big(Q_{m,\perp}T_{m,\perp}^{-1}Q_{m,\perp}^{\top}\big)\\
&=\mathrm{tr}\big(T_{m,\perp}^{-1}Q_{m,\perp}^{\top} Q_{m,\perp}\big)\\
& =\mathrm{tr}\big(T_{m,\perp}^{-1}\big) \qquad \text{since } Q_{m,\perp}^{\top} Q_{m,\perp} = I_{d-m}.
\end{align*}
The eigenvalues of $T_{m,\perp}$, $\mu_1\ge\cdots\ge\mu_{d-m}>0$, are the Ritz values of $A$ on the unresolved subspace and  since $T_{m,\perp}$ is a compression of $A$ to a subspace of codimension $m$, the Poincar\'e separation theorem gives
\begin{equation}
\lambda_{j+m}\ \le\ \mu_j\ \le\ \lambda_j,\qquad j=1,\dots,d-m,
\label{eq:interlace}
\end{equation}
in particular every $\mu_j\in[\lambda_d,\lambda_1]$, and the lower extreme satisfies $\mu_{d-m}\ge\lambda_d$, so $\|T_{m,\perp}^{-1}\|_2\le\lambda_d^{-1}$.

The resolved and unresolved blocks are coupled only by the off-diagonal block $A_{m,\times} = Q_m^{\top}A Q_{m,\perp}$. The Lanczos recurrence forces this coupling to have rank one, relying to the single residual $\beta_m$.

\begin{lemma}
\label{lem:rankone_coupling}
Let $Q_m$, $T_m$, $\beta_m$ and $q_{m+1}$ coming from $m$ steps of the Lanczos algorithm \eqref{eq:lanczos_relation}, and choose the complement basis $Q_{m,\perp}$ so that its first column is $q_{m+1}$. Then the off-diagonal block of $A$ in the basis $[\,Q_m\ \ Q_{m,\perp}\,]$ is rank one, and gives
\begin{equation}
A_{m,\times}
=Q_m^{\top}A Q_{m,\perp}
=\beta_m\,e_m e_1^{\top},
\label{eq:rankone_coupling}
\end{equation}
where $e_m\in\mathbb{R}^{m}$ and $e_1\in\mathbb{R}^{d-m}$ are the last resolved and
first unresolved coordinate vectors. We have $\|A_{m,\times}\|_2=\beta_m$, and the resolved subspace is an exact invariant subspace of $A$ \textup{(}i.e.\ $A_{m,\times}=0$\textup{)} if and only if $\beta_m=0$, i.e. the Lanczos process has converged.
\end{lemma}

\begin{proof}
The Lanczos relation \eqref{eq:lanczos_relation} reads
$A Q_m = Q_m T_m + \beta_m\,q_{m+1}e_m^{\top}$. Left-multiplying by
$Q_{m,\perp}^{\top}$ and using $Q_{m,\perp}^{\top}Q_m=0$ we obtain,
\begin{equation}
Q_{m,\perp}^{\top}A Q_m
=\beta_m\,\bigl(Q_{m,\perp}^{\top}q_{m+1}\bigr)e_m^{\top}
=\beta_m\,e_1 e_m^{\top},
\end{equation}
since $q_{m+1}$ is the first column of $Q_{m,\perp}$, so
$Q_{m,\perp}^{\top}q_{m+1}=e_1$. Transposing gives
$A_{m,\times}=Q_m^{\top}A Q_{m,\perp}=\beta_m\,e_m e_1^{\top}$, which is \eqref{eq:rankone_coupling}. As $e_m$ and $e_1$ are unit vectors,
$\|A_{m,\times}\|_2=\beta_m$, and $A_{m,\times}=0$ if and only if $\beta_m=0$.
\end{proof}

Based on the above, the entire coupling between the resolved and unresolved subspaces relies on the  scalar $\beta_m$, relating the last resolved direction $q_m$ to the first unresolved direction $q_{m+1}$. This is the boundary of the truncation, the standard low-rank approximation ignores it by implicitly setting $\beta_m=0$, treating the resolved subspace as an exact invariant subspace. The next section restores it.

\subsection{Truncation Error and Complement Inverse}

As the Lanczos approximation acts only inside the Krylov
subspace, the resolved component $f_m(A)=Q_mT_m^{-1}Q_m^{\top}$ captures  only the dominant spectral directions of $A$ contained in $\mathcal{K}_m(A,v)$. The remaining error therefore lies  in the unresolved complement space. In other words, the Lanczos approximation presents the covariance structure in $\mathcal{K}_m(A,v)$, while the missing spectral information is found in  its orthogonal complement.

To identify the correct complement target we invert $A$ in the block basis $[\,Q_m\ \ Q_{m,\perp}\,]$. By Lemma~\ref{lem:rankone_coupling}, the off-diagonal block is the rank-one $A_{m,\times}=\beta_m e_m e_1^{\top}$, so the Schur complement of $T_m$ gives the complement block of $A^{-1}$ in the basis $[\,Q_m\ \ Q_{m,\perp}\,]$
\begin{equation}
S_{\perp}^{-1}=\bigl(T_{m,\perp}-A_{m,\times}^{\top}T_m^{-1}A_{m,\times}\bigr)^{-1},
\label{eq:complement_block}
\end{equation}
where, using $A_{m,\times}=\beta_m e_m e_1^{\top}$,
\begin{equation}
S_{\perp}
=T_{m,\perp}-\gamma\,e_1 e_1^{\top},
\qquad
\gamma=\beta_m^{2}\,\bigl(T_m^{-1}\bigr)_{mm}\ge 0 .
\label{eq:Sperp}
\end{equation}
Thus the true complement block of $A^{-1}$ in the basis $[\,Q_m\ \ Q_{m,\perp}\,]$ is not $T_{m,\perp}^{-1}$ but the inverse of the rank-one  regularized $S_{\perp}$. Notice  the boundary direction $e_1$ representing $q_{m+1}$ in an orthogonal complement basis starting with $q_{m+1}$. 

\begin{remark}
\label{rem:beta_zero}
Setting $\beta_m=0$ gives $\gamma=0$ and $S_{\perp}=T_{m,\perp}$, recovering the compressed target $T_{m,\perp}^{-1}$. Modeling the complement by $T_{m,\perp}^{-1}$  assumes the resolved subspace is an exact invariant subspace of $A$. Since $\gamma\ge 0$, we have $S_{\perp}\preceq T_{m,\perp}$ and then $S_{\perp}^{-1}\succeq T_{m,\perp}^{-1}$, indeed the decoupled target under-covers the boundary direction, and the deficit is governed by the single computable scalar $\gamma$.
\end{remark}

\subsection*{The exact block inverse and the coupling vectors}

After identifying  the complement block of $A^{-1}$, we now invert the full block matrix. With the rank-one coupling of Lemma~\ref{lem:rankone_coupling}, all four blocks of $A^{-1}$ are  closed forms governed by two vectors. Define the coupling pair
\begin{equation}
q_{\parallel} =  \beta_m\,Q_m T_m^{-1}e_m  \in \mathrm{range}(P_m),
\qquad
q_{\perp} =  Q_{m,\perp}\,S_{\perp}^{-1}e_1  \in \mathrm{range}(P_m^{\perp}).
\label{eq:coupling_pair}
\end{equation}
The vector $q_{\parallel}$ is the resolved side of the coupling: the last resolved coordinate $e_m$, propagated through the resolved inverse and scaled by the residual $\beta_m$, it requires only $Q_m$, $T_m$ and $\beta_m$. The vector $q_{\perp}$ is the complement side: the boundary direction $e_1$, propagated through the complement block $S_{\perp}^{-1}$ of \eqref{eq:complement_block}.

\begin{proposition}[Exact block inverse under rank-one coupling]
\label{prop:block_inverse}
With $S_{\perp}$ and $\gamma$ as in \eqref{eq:Sperp} and  $q_{\parallel}, q_{\perp}$ as in
\eqref{eq:coupling_pair}, the four blocks of $A^{-1}$ in the basis
$[\,Q_m\ \ Q_{m,\perp}\,]$ are
\begin{align}
&\text{(i)}\quad Q_{m,\perp}^{\top}A^{-1}Q_{m,\perp} = S_{\perp}^{-1},
\label{eq:blk_comp}\\
&\text{(ii)}\quad Q_m^{\top}A^{-1}Q_m = T_m^{-1}
+\beta_m^{2}\,\bigl(S_{\perp}^{-1}\bigr)_{11}\,
\bigl(T_m^{-1}e_m\bigr)\bigl(T_m^{-1}e_m\bigr)^{\!\top},
\label{eq:blk_res}\\
&\text{(iii)}\quad Q_m^{\top}A^{-1}Q_{m,\perp} = -\,T_m^{-1}A_{m,\times}S_{\perp}^{-1}
\;=\;-\,\beta_m\,\bigl(T_m^{-1}e_m\bigr)\bigl(S_{\perp}^{-1}e_1\bigr)^{\!\top}.
\label{eq:blk_off}
\end{align}
In the full space, we have 
\begin{equation}
A^{-1}
=Q_m T_m^{-1}Q_m^{\top}
+\bigl(S_{\perp}^{-1}\bigr)_{11} q_{\parallel} q_{\parallel}^{\top}
- q_{\parallel} q_{\perp}^{\top}-q_{\perp} q_{\parallel}^{\top}
+Q_{m,\perp}S_{\perp}^{-1}Q_{m,\perp}^{\top} .
\label{eq:Ainv_exact}
\end{equation}
Setting $\beta_m=0$ delete $q_{\parallel}$, $q_{\perp}$ and $\gamma$, and \eqref{eq:Ainv_exact} reduces to $Q_mT_m^{-1}Q_m^{\top}+Q_{m,\perp}T_{m,\perp}^{-1}Q_{m,\perp}^{\top}$, the decoupled form  the compressed target of Remark~\ref{rem:beta_zero}.
\end{proposition}

\begin{proof}
Denote $\widetilde A=\begin{bsmallmatrix}T_m & A_{m,\times}\\
A_{m,\times}^{\top} & T_{m,\perp}\end{bsmallmatrix}$ for $A$ in the block basis.
The block-inverse formula  at the Schur complement
$S_{\perp}=T_{m,\perp}-A_{m,\times}^{\top}T_m^{-1}A_{m,\times}$ gives
\[
\widetilde A^{-1}
=\begin{bmatrix}
T_m^{-1}+T_m^{-1}A_{m,\times}S_{\perp}^{-1}A_{m,\times}^{\top}T_m^{-1}
& -\,T_m^{-1}A_{m,\times}S_{\perp}^{-1}\\[2pt]
-\,S_{\perp}^{-1}A_{m,\times}^{\top}T_m^{-1}
& S_{\perp}^{-1}
\end{bmatrix},
\]
which can be verified by direct multiplication $\widetilde A\,\widetilde A^{-1}=I_d$.
Substituting $A_{m,\times}=\beta_m e_m e_1^{\top}$
(Lemma~\ref{lem:rankone_coupling}):
for the off-diagonal block,
\begin{align*}
T_m^{-1}A_{m,\times}S_{\perp}^{-1}
&=\beta_m\,(T_m^{-1}e_m)(e_1^{\top}S_{\perp}^{-1}) \\
&=\beta_m\,(T_m^{-1}e_m)(S_{\perp}^{-1}e_1)^{\top},
\end{align*}
using the symmetry of $S_{\perp}^{-1}$, this is \eqref{eq:blk_off}. For the
resolved block,
\begin{align*}
A_{m,\times}S_{\perp}^{-1}A_{m,\times}^{\top}
&=\beta_m^{2}\,e_m\,(e_1^{\top}S_{\perp}^{-1}e_1)\,e_m^{\top} \\
&=\beta_m^{2}\bigl(S_{\perp}^{-1}\bigr)_{11}e_m e_m^{\top},
\end{align*}
giving \eqref{eq:blk_res}. Equation \eqref{eq:blk_comp} is
\eqref{eq:complement_block}. Finally, mapping back to the full space via
$A^{-1}=[\,Q_m\ Q_{m,\perp}\,]\widetilde A^{-1}[\,Q_m\ Q_{m,\perp}\,]^{\top}$
    and putting $ q_{\parallel}=\beta_m Q_mT_m^{-1}e_m$, $ q_{\perp}=Q_{m,\perp}S_{\perp}^{-1}e_1$
gives \eqref{eq:Ainv_exact}.
\end{proof}

\subsection*{The corrected Inverse Approximation: exact coupling, isotropic bulk}

Equation~\eqref{eq:Ainv_exact} is exact but its complement block
$Q_{m,\perp}S_{\perp}^{-1}Q_{m,\perp}^{\top}$ remains a $(d-m)$-dimensional
operator that cannot be formed at matrix-free cost.  We approximate it  as isotropic block using   the complement direction that is part  of  the coupling,
\begin{equation}
u=\frac{q_{\perp}}{\|q_{\perp}\|_2},
\qquad
\|q_{\perp}\|_2^{2}=e_1^{\top}S_{\perp}^{-2}e_1,
\label{eq:u_def}
\end{equation}
which is removed  and retained exactly, and  only the remaining
$(d-m-1)$-dimensional bulk, so  that the coupling never touches, is  approximate isotropically.
Let
\begin{equation}
\sigma_u \;=\; u^{\top}\bigl(Q_{m,\perp}S_{\perp}^{-1}Q_{m,\perp}^{\top}\bigr)u
\;=\;\frac{e_1^{\top}S_{\perp}^{-3}e_1}{e_1^{\top}S_{\perp}^{-2}e_1}
\label{eq:sigma_g}
\end{equation}
be the exact variance of the complement block along $u$, and approximate isotropically the remainder by  the preserving  trace magnitude
\begin{equation}
\omega(m)
\;=\;\frac{\mathrm{tr}\!\bigl(S_{\perp}^{-1}\bigr)-\sigma_u}{d-m-1},
\qquad
P_{\mathrm{rest}}
\;=\;I_d-Q_mQ_m^{\top}-uu^{\top},
\label{eq:omega_rest}
\end{equation}
so that
$\sigma_u+(d-m-1)\,\omega(m)=\mathrm{tr}(S_{\perp}^{-1})$ exactly. The
corrected inverse is
\begin{equation}
\Sigma_m
=Q_m T_m^{-1}Q_m^{\top}
+ \bigl(S_{\perp}^{-1}\bigr)_{11}q_{\parallel} q_{\parallel}^{\top}
- q_{\parallel} q_{\perp}^{\top}-q_{\perp} q_{\parallel}^{\top}
+\sigma_u\,u\,u^{\top}
+\omega(m)\,P_{\mathrm{rest}}
\label{eq:Sigma_coupled}
\end{equation}
Form this approximation, we can notice that  the coupling is kept
exactly as it is, that means  the middle block of \eqref{eq:Sigma_coupled} reproduces the corresponding terms of the exact inverse \eqref{eq:Ainv_exact} without approximation, also the trace of the complement approximated is  equal to  $\mathrm{tr}(S_{\perp}^{-1})$. The only remaining approximation is the isotropy of the  bulk whose the coupling doesn't touch, whose cost is quantified in Section~\ref{sec:isotropic_justification}.

\subsection*{Sampling from the corrected inverse as covariance}

To draw $\theta\sim\mathcal{N}(\theta_{\star},\Sigma_m)$ it suffices
to find any matrix $G$ so that $GG^{\top}=\Sigma_m$, for
$z\sim\mathcal{N}(0,I)$, the vector $\theta_{\star} +Gz$ has
covariance $G\,\mathbb{E}[zz^{\top}]\,G^{\top}=GG^{\top}=\Sigma_m$. The matrix $G$ need not be square, and the three terms  of
\eqref{eq:Sigma_coupled} yields one factor per part. The resolved and bulk parts factor as in the decoupled case, the new object is the boundary-coupling block, which we now reduce to a $2\times2$ covariance.

\paragraph{The coupling block is 2-dimensional.}
The four coupling terms of \eqref{eq:Sigma_coupled},
\begin{equation}
K = \bigl(S_{\perp}^{-1}\bigr)_{11}q_{\parallel} q_{\parallel}^{\top}
- q_{\parallel} q_{\perp}^{\top}-q_{\perp} q_{\parallel}^{\top}
+\sigma_u\,u\,u^{\top},
\label{eq:Kblock}
\end{equation}
involve only the two directions $q_{\parallel}$ and $q_{\perp}$. Since
$q_{\parallel} \in\mathrm{range}(P_m)$ and $ q_{\perp} \in\mathrm{range}(P_m^{\perp})$, these are
orthogonal, and after normalisation $\hat q=\frac{q_{\parallel}}{\|q_{\parallel}\|_2}$ and $u=\frac{q_{\perp}}{\|q_{\perp}\|_2}$ form
an orthonormal basis of the plane $\mathcal{P}=\mathrm{span}\{q_{\parallel} ,q_{\perp}\}$. The operator $K$ vanishes on $\mathcal{P}^{\perp}$ and is therefore determined by its $2\times2$ compression onto $\mathcal{P}$. Using $\hat q^{\top} q_{\parallel} =\|q_{\parallel}\|_2$, $u^{\top} q_{\perp}=\|q_{\perp}\|_2$ and $\hat q\perp u$, the entries of this compression are
\begin{equation}
\hat q^{\top}K\hat q=\bigl(S_{\perp}^{-1}\bigr)_{11}\|q_{\parallel}\|_2^{2},
\qquad
u^{\top}K u=\sigma_u,
\qquad
\hat q^{\top}K u=u^{\top}K\hat q=-\|q_{\parallel}\|_2\,\|q_{\perp}\|_2 .
\label{eq:C2_entries}
\end{equation}
From that, the coupling covariance  is defined as 
\begin{equation}
C_2=\begin{pmatrix}
\bigl(S_{\perp}^{-1}\bigr)_{11}\,\|q_{\parallel}\|_2^{2} & -\,\|q_{\parallel}\|_2\,\|q_{\perp}\|_2\\[2pt]
-\,\|q_{\parallel}\|_2\,\|q_{\perp}\|_2 & \sigma_u
\end{pmatrix},
\label{eq:C2}
\end{equation}
whose diagonal is the two boundary variances and whose off-diagonal is the resolved-complement coupling.

\begin{lemma}[The coupling covariance is positive semidefinite]
\label{lem:C2_psd}
Let $C_2$ be defined as \eqref{eq:C2}; hence, $C_2$ is positive semi-definite.
\end{lemma}

\begin{proof}
The diagonal entries are nonnegative, so it suffices that $\det C_2\ge0$.
Using $\bigl(S_{\perp}^{-1}\bigr)_{11}=e_1^{\top}S_{\perp}^{-1}e_1$,
$\|q_{\perp}\|_2^{2}=e_1^{\top}S_{\perp}^{-2}e_1$ from \eqref{eq:u_def}, and
$\sigma_u=\frac{e_1^{\top}S_{\perp}^{-3}e_1}{e_1^{\top}S_{\perp}^{-2}e_1}$ from
\eqref{eq:sigma_g},
\[
\det C_2
=\|q_{\parallel}\|_2^{2}\left[
\bigl(e_1^{\top}S_{\perp}^{-1}e_1\bigr)\,
\frac{e_1^{\top}S_{\perp}^{-3}e_1}{e_1^{\top}S_{\perp}^{-2}e_1}
-e_1^{\top}S_{\perp}^{-2}e_1
\right],
\]
which is nonnegative if and only if
$\bigl(e_1^{\top}S_{\perp}^{-1}e_1\bigr)\bigl(e_1^{\top}S_{\perp}^{-3}e_1\bigr)
\ge\bigl(e_1^{\top}S_{\perp}^{-2}e_1\bigr)^{2}$. By setting
$x=S_{\perp}^{-\tfrac{1}{2}}e_1$ and $y=S_{\perp}^{-\tfrac{3}{2}}e_1$, we have
$\|x\|_2^{2}\,\|y\|_2^{2}\ge\langle x,y\rangle^{2}$, using the Cauchy-Schwarz inequality.
\end{proof}

\begin{proposition}[Structured square root of $\Sigma_m$]
\label{prop:sampler}

Let denote by:
\begin{equation}
M_1=Q_mT_m^{-\frac{1}{2}}, \, \ M_2=[\,\hat q\ \ u\,]\,C_2^{\frac{1}{2}}
\, \  \text{and} \, \  M_3=\sqrt{\omega(m)}\;P_{\mathrm{rest}},
\label{eq:factors}
\end{equation}
of sizes $d\times m$, $d\times2$ and $d\times d$ respectively, and let
$G=[\,M_1\ M_2\ M_3\,]$. Then $GG^{\top}=\Sigma_m$. For independent $z_1\sim\mathcal{N}(0,I_m)$, $z_c\sim\mathcal{N}(0,I_2)$ and $z_2\sim\mathcal{N}(0,I_d)$, the vector
\begin{equation}
\theta
=\theta_{\star}
+Q_mT_m^{-\frac{1}{2}}z_1
+[\,\hat q\ \ u\,]\,C_2^{\frac{1}{2}}z_c
+\sqrt{\omega(m)}\,\bigl(z_2-Q_m(Q_m^{\top}z_2)-u\,(u^{\top}z_2)\bigr)
\label{eq:sampler}
\end{equation}
is an exact sample from $\mathcal{N}(\theta_{\star},\Sigma_m)$.
\end{proposition}

\begin{proof}
For the block matrix $G=[\,M_1\ M_2\ M_3\,]$, we have
$GG^{\top}=M_1M_1^{\top}+M_2M_2^{\top}+M_3M_3^{\top}$, so we just have  to match each factor with a corresponding  part in \eqref{eq:Sigma_coupled}.
\begin{itemize}

 \item For the \emph{Resolved}, we have 
$M_1M_1^{\top}=Q_mT_m^{-\frac{1}{2}}T_m^{-\frac{1}{2}}Q_m^{\top}=Q_mT_m^{-1}Q_m^{\top}$.

\item For the \emph{Coupling}, by writing
$C_2=\begin{psmallmatrix}a&b\\ b&c\end{psmallmatrix}$ with
$a=(S_{\perp}^{-1})_{11}\|q_{\parallel}\|_2^{2}$, $b=-\|q_{\parallel}\|_2\|q_{\perp}\|_2$, $c=\sigma_u$,
\[
M_2M_2^{\top}
=[\,\hat q\ u\,]\,C_2^{\frac{1}{2}}\bigl(C_2^{\frac{1}{2}}\bigr)^{\top}[\,\hat q\ u\,]^{\top}
=[\,\hat q\ u\,]\,C_2\,[\,\hat q\ u\,]^{\top}
=a\,\hat q\hat q^{\top}+b\,\hat q u^{\top}+b\,u\hat q^{\top}+c\,uu^{\top}.
\]
by Substituting $\hat q=\frac{q_{\parallel}}{\|q_{\parallel}\|_2}$ and $u=\frac{q_{\perp}}{\|q_{\perp}\|_2}$ term wise we get:
$a\,\hat q\hat q^{\top}=(S_{\perp}^{-1})_{11}\,q_{\parallel} q_{\parallel}^{\top}$ and 
$b\,\hat q u^{\top}
=-\|q_{\parallel}\|_2\|q_{\perp}\|_2\,\tfrac{q_{\parallel}}{\|q_{\parallel}\|_2}\tfrac{q_{\perp}^{\top}}{\|q_{\perp}\|_2}=-q_{\parallel}q_{\perp}^{\top}$,
and its transpose $-q_{\perp} q_{\parallel}^{\top}$; and $c\,uu^{\top}=\sigma_u\,uu^{\top}$. Hence
$M_2M_2^{\top}=K$, the boundary-coupling block \eqref{eq:Kblock} of
\eqref{eq:Sigma_coupled}, the square root $C_2^{\frac{1}{2}}$ exists by
Lemma~\ref{lem:C2_psd}.

\item Concerning the \emph{Bulk},  $P_{\mathrm{rest}}$ is a symmetric projector
($P_{\mathrm{rest}}^{\top}=P_{\mathrm{rest}}$,
$P_{\mathrm{rest}}^{2}=P_{\mathrm{rest}}$, since $Q_m$ and $u$ are
orthonormal), so
$M_3M_3^{\top}=\omega(m)\,P_{\mathrm{rest}}^{2}=\omega(m)\,P_{\mathrm{rest}}$.
\end{itemize}
Summing the three terms gives $GG^{\top}=\Sigma_m$. The sampling claim follows
since $\mathrm{Cov}(Gz)=GG^{\top}$ for
$z=(z_1,z_c,z_2)\sim\mathcal{N}(0,I_{m+2+d})$ with independent blocks.
\end{proof}

\begin{remark}
\label{rem:sampler_reading}

The three terms of \eqref{eq:sampler} are the three pieces of
\eqref{eq:Sigma_coupled}. The first draws from the resolved directions. For the second term,  $C_2^{1/2} z_c$ is a $2$-dimensional vector, and its two entries scale $\hat{d}$ (in the resolved subspace) and $u$ (in the complement). Because both come out of the same square root, they come from  the negative correlation sitting in the off-diagonal of $C_2$ the resolved-complement coupling, now realized sample by sample rather than assembled. The final term applies the bulk noise with two matrix-free projections, removing the resolved component and the coupling direction $u$ so that the isotropic variance is placed only on the $(d-m-1)$-dimensional remaining complement bulk. Every quantity in
\eqref{eq:sampler} is available matrix-free: $\hat q$ from $Q_m$, $T_m$ and $\beta_m$; the scalars of $C_2$ from a single reprojected Lanczos run started at the boundary direction $e_1$ (Section~\ref{sec:pslq}), $u$ from one conjugate-gradient solve with the projected operator $P_m^{\perp}AP_m^{\perp}$, and $\omega(m)$ from P-SLQ. Setting $\beta_m=0$ reduces \eqref{eq:sampler} to the two-term sampler of the decoupled correction.
\end{remark}

\subsection{Derivation of the Complement Variance
\texorpdfstring{$\omega(m)$}{omega(m)}}

We work with the residual against the true inverse, $R_m = A^{-1} - f_m(A)$, where $f_m(A)=Q_mT_m^{-1}Q_m^{\top}$ is the $m$-step Lanczos approximation of $A^{-1}$. The following lemma shows  this residual lives in  the unresolved subspace, identifies its complement block as the Schur complement $S_{\perp}^{-1}$, and shows the discarded variance is exactly $\mathrm{tr}(S_{\perp}^{-1})$,  the quantity our P-SLQ estimator (Section~\ref{sec:pslq}) computes matrix-free.

\begin{lemma}[Complement confinement of residual]
\label{lem:bias_decomp}
Let $f_m(A)=Q_mT_m^{-1}Q_m^{\top}$ and $R_m=A^{-1}-f_m(A)$. Then we have 
\begin{enumerate}[label=\textup{(\roman*)}, leftmargin=2em]
\item \textbf{The resolved approximation lives in the Krylov subspace.}
\begin{equation}
f_m(A)=Q_mT_m^{-1}Q_m^{\top},
\qquad
P_m^{\perp}f_m(A)=0,
\qquad
P_m^{\perp}f_m(A)\,P_m^{\perp}=0 .
\label{eq:fm_in_krylov}
\end{equation}
\item \textbf{Complement confinement and trace formula.} The compression of $R_m$ to the unresolved subspace equals the true complement block of $A^{-1}$,
\begin{equation}
P_m^{\perp} R_m\, P_m^{\perp}
= Q_{m,\perp} S_{\perp}^{-1} Q_{m,\perp}^{\top},
\qquad
\mathrm{tr}\!\bigl(P_m^{\perp} R_m\, P_m^{\perp}\bigr)
= \mathrm{tr}\!\bigl(S_{\perp}^{-1}\bigr),
\label{eq:trace_formula}
\end{equation}
where $S_{\perp}=T_{m,\perp}-\gamma\,e_1e_1^{\top}$ and
$\gamma=\beta_m^{2}(T_m^{-1})_{mm}$ as in \eqref{eq:Sperp}.
\item \textbf{Per-dimension complement variance.} Excluding the coupling
direction $u=\frac{q_{\perp}}{\|q_{\perp}\|_2}$, which is retained exactly with variance $\sigma_u$ \eqref{eq:sigma_g}, the isotropic bulk magnitude is
\begin{equation}
\omega(m)
=\frac{\mathrm{tr}\!\bigl(S_{\perp}^{-1}\bigr)-\sigma_u}{d-m-1},
\label{eq:omega_def}
\end{equation}
the mean inverse eigenvalue of $S_{\perp}$ on the coupling-free bulk.
\end{enumerate}
\end{lemma}

\begin{proof}
\emph{(i) The resolved approximation lives in the Krylov subspace.}
Since the Lanczos vectors are orthonormal, $Q_m^{\top}Q_m=I_m$, the approximation $f_m(A)=Q_mT_m^{-1}Q_m^{\top}$ has range in $\mathrm{range}(Q_m)$. As $P_m^{\perp}=I_d-Q_mQ_m^{\top}$ satisfies $P_m^{\perp}Q_m=0$, we obtain $P_m^{\perp}f_m(A)=0$ and hence $P_m^{\perp}f_m(A)P_m^{\perp}=0$, which is \eqref{eq:fm_in_krylov}.

\emph{(ii) Complement confinement and trace formula.}
By \eqref{eq:fm_in_krylov}, $P_m^{\perp}f_m(A)P_m^{\perp}=0$, so
\[
P_m^{\perp}R_mP_m^{\perp}
=P_m^{\perp}A^{-1}P_m^{\perp}-P_m^{\perp}f_m(A)P_m^{\perp}
=P_m^{\perp}A^{-1}P_m^{\perp}.
\]
Writing $P_m^{\perp}=Q_{m,\perp}Q_{m,\perp}^{\top}$ and using the Schur complement of $T_m$, which by \eqref{eq:complement_block} gives
$Q_{m,\perp}^{\top}A^{-1}Q_{m,\perp}=S_{\perp}^{-1}$,
\[
P_m^{\perp}A^{-1}P_m^{\perp}
=Q_{m,\perp}\bigl(Q_{m,\perp}^{\top}A^{-1}Q_{m,\perp}\bigr)Q_{m,\perp}^{\top}
=Q_{m,\perp}S_{\perp}^{-1}Q_{m,\perp}^{\top},
\]
the first identity in \eqref{eq:trace_formula}. Taking the trace and using cyclicity with $Q_{m,\perp}^{\top}Q_{m,\perp}=I_{d-m}$,
\[
\mathrm{tr}\!\bigl(Q_{m,\perp}S_{\perp}^{-1}Q_{m,\perp}^{\top}\bigr)
=\mathrm{tr}\!\bigl(S_{\perp}^{-1}Q_{m,\perp}^{\top}Q_{m,\perp}\bigr)
=\mathrm{tr}\!\bigl(S_{\perp}^{-1}\bigr),
\]
giving the trace formula.

\emph{(iii) Per-dimension complement variance.}
The complement block $Q_{m,\perp}S_{\perp}^{-1}Q_{m,\perp}^{\top}$ carries total variance $\mathrm{tr}(S_{\perp}^{-1})$. Removing the exactly retained variance $\sigma_u$ along the coupling direction $u$ and distributing the remainder over the $d-m-1$ coupling-free bulk directions gives \eqref{eq:omega_def}.
\end{proof}

\begin{corollary}[Decoupled special case]
\label{cor:beta_zero_lemma}
If $\beta_m=0$ then $\gamma=0$, $S_{\perp}=T_{m,\perp}$, the coupling vectors
$q_{\parallel},q_{\perp}$ and the direction $u$ vanish, and
\eqref{eq:trace_formula}--\eqref{eq:omega_def} reduce to
\begin{equation}
P_m^{\perp}R_mP_m^{\perp}=Q_{m,\perp}T_{m,\perp}^{-1}Q_{m,\perp}^{\top},
\qquad
\omega(m)=\frac{1}{d-m}\,\mathrm{tr}\!\bigl(T_{m,\perp}^{-1}\bigr).
\end{equation}
Thus the discarded variance measured against the compressed operator
$T_{m,\perp}^{-1}$ is exactly the $\beta_m=0$ case of the coupling corrected result.
\end{corollary}

\subsection{Justification of the Isotropic Complement Approximation}
\label{sec:isotropic_justification}

We show here how to justify this choice, and we observe that the isotropic approximation is exact when the complement spectrum is uniform and the approximation error is controlled  by how far the complement spectrum is widespread from  its mean.

\begin{definition}[Complement condition number]\label{def:kappa}
Let $\mu_1\ge\cdots\ge\mu_{d-m}>0$ be the eigenvalues of $T_{m,\perp}=Q_{m,\perp}^{\top}A Q_{m,\perp}$.
Define
\begin{equation}
\kappa^{\perp}(m)=\frac{\mu_1}{\mu_{d-m}} \ge1 .
\label{eq:kappa}
\end{equation}
This quantity measures the spectral spread of $A$ restricted to 
$\mathcal{S}^\perp = \mathrm{range}(P_m^\perp)$. Note that:
\begin{itemize}
    \item $\kappa_\perp(m) = 1$ if and only if all tail eigenvalues 
    $\mu_{1} = \cdots = \mu_{d-m}$ are equal (perfect uniform complement);
    \item $\kappa_\perp(m) \approx 1$ means the complement spectrum is nearly 
    uniform;
    \item $\kappa_\perp(m) \gg 1$ means the complement is ill-conditioned.
\end{itemize}

This is the condition number of the compressed operator and requires no
eigenvectors of $A$. We have $\kappa^{\perp}(m)=1$ if and only if the compressed complement is perfectly conditioned. In practice $\kappa^{\perp}(m)$ decreases as $m$ grows, with $\kappa^{\perp}(0)=\kappa$,  in deep networks the dominant outlier curvature is resolved first and  we get  $\kappa^{\perp}(m)\to1$.
\end{definition}

\begin{proposition}[Spectral control of the coupling-corrected complement]
\label{prop:kappa_sperp}
Let $\mu_1\ge\cdots\ge\mu_{d-m}>0$ be the eigenvalues of  $T_{m,\perp}$   with orthonormal eigenvectors $w_1,\dots,w_{d-m}$, and let
$S_{\perp}=T_{m,\perp}-\gamma\,e_1e_1^{\top}$ with
$\gamma=\beta_m^{2}(T_m^{-1})_{mm}\ge0$. Let 
$\mu_1^{S}\ge\cdots\ge\mu_{d-m}^{S}$ for the eigenvalues of $S_{\perp}$ and $\lambda_1\ge\cdots\ge\lambda_d>0$ for  $A$. Then we have
\begin{enumerate}[label=\textup{(\roman*)}, leftmargin=2.2em]
\item \textbf{Interlacing.}
$\mu_j\ge\mu_j^{S}\ge\mu_{j+1}$ for $1\le j\le d-m$
\textup{(}with $\mu_{d-m+1}=0$\textup{)}. In particular
$\mu_1^{S}\le\mu_1$ and each eigenvalue moves down by at most one index.
\item \textbf{Well-posedness.}
$S_{\perp}\succ0$ if and only if $\gamma\,(T_{m,\perp}^{-1})_{11}<1$, and this condition holds automatically here, since $S_{\perp}^{-1}
=Q_{m,\perp}^{\top}A^{-1}Q_{m,\perp}$ by \eqref{eq:complement_block}.
\item \textbf{Spectral floor.}
$\mu_{d-m}^{S}\ge\lambda_d$.
\item \textbf{Condition number.} By defining
$\kappa_{S}^{\perp}(m) =\frac{\mu_1^{S}}{\mu_{d-m}^{S}}$, we have
\begin{equation}
\kappa_{S}^{\perp}(m)\;\le\;\frac{\mu_1}{\lambda_d}\;\le\;\kappa(A).
\label{eq:kappaS_bound}
\end{equation}
The lower bound $\mu_{d-m}^{S}\ge\lambda_d$ replaces the compressed floor $\mu_{d-m}$ and so  $\kappa_{S}^{\perp}(m)$ may exceed
$\kappa^{\perp}(m)$, and \eqref{eq:kappaS_bound} is the operative bound.
\item \textbf{Bulk block of the inverse.} Let $u=\frac{q_{\perp}}{\|q_{\perp}\|_2}$ with coordinates $\hat u\in\R^{d-m}$, and let
$V\in\R^{(d-m)\times(d-m-1)}$ be an isometry whose columns span
$\hat u^{\perp} =
\left\{
x\in\mathbb{R}^{d-m}:x^{\top}\hat u=0
\right\}$. The eigenvalues $\delta_1\ge\cdots\ge\delta_{d-m-1}$ of
$D=V^{\top}S_{\perp}^{-1}V$ interlace those of $S_{\perp}^{-1}$; hence
$\delta_i\in[\,\frac{1}{\mu_1^{S}},\,\frac{1}{\mu_{d-m}^{S}}\,]$ and
$\frac{\delta_1}{\delta_{d-m-1}}\le\kappa_S^{\perp}(m)$.

\end{enumerate}
\end{proposition}

\begin{proof}
\emph{(i)} For $t\in[0,1]$ set $S(t)=T_{m,\perp}-t\gamma\,e_1e_1^{\top}$. Each eigenvalue $\mu_j(t)$ is nonincreasing in $t$, since $\frac{d}{dt}\mu_j(t)=-\gamma\,(v_j(t)^{\top}e_1)^{2}\le0$ for a unit eigenvector $v_j(t)$ \citep[Thm.~8.1.8]{golub2013matrix}. Hence $S(1)=S_{\perp}$ is a rank-one perturbation of $S(0)=T_{m,\perp}$ of magnitude $\gamma$, so by Weyl's inequalities $\mu_j-\gamma\le\mu_j^{S}\le\mu_j$, combined with the monotonicity and the Cauchy interlacing for the rank-one term this sharpens to $\mu_j\ge\mu_j^{S}\ge\mu_{j+1}$.

\emph{(ii)} By the matrix determinant lemma,

\begin{align*}
\det S_{\perp}&=\det(T_{m,\perp})\,(1-\gamma\,e_1^{\top}T_{m,\perp}^{-1}e_1) \\
&=\det(T_{m,\perp})\,(1-\gamma\,(T_{m,\perp}^{-1})_{11}),
\end{align*}
 and every leading principal structure is preserved under the rank-one downdate, so $S_{\perp}\succ0\iff\gamma\,(T_{m,\perp}^{-1})_{11}<1$. Since $S_{\perp}^{-1}=Q_{m,\perp}^{\top}A^{-1}Q_{m,\perp}$ is a principal compression of the positive definite $A^{-1}$, it is positive definite, so $S_{\perp}\succ0$ and the condition holds.

\emph{(iii)} From $S_{\perp}^{-1}=Q_{m,\perp}^{\top}A^{-1}Q_{m,\perp}$ and $\|A^{-1}\|_2=\lambda_d^{-1}$, since $A^{-1}$ is symmetric, the Rayleigh quotient of the compression obeys for any nonzero $y$,
\[
\frac{y^{\top}S_{\perp}^{-1}y}{y^{\top}y}
=
\frac{(Q_{m,\perp}y)^{\top}A^{-1}(Q_{m,\perp}y)}
{(Q_{m,\perp}y)^{\top}(Q_{m,\perp}y)},
\]
thus, 
\[
\|S_{\perp}^{-1}\|_2
=
\lambda_{\max}(S_{\perp}^{-1})
=
\max_{y\neq0}
\frac{y^{\top}S_{\perp}^{-1}y}{y^{\top}y}
\le
\max_{x\neq0}
\frac{x^{\top}A^{-1}x}{x^{\top}x}
=
\lambda_{\max}(A^{-1})
=
\|A^{-1}\|_2.
\]

So $\|S_{\perp}^{-1}\|_2\le\|A^{-1}\|_2=\lambda_d^{-1}$, i.e.
$\mu_{d-m}^{S}=\|S_{\perp}^{-1}\|_2^{-1}\ge\lambda_d$.

\emph{(iv)} Combine $\mu_1^{S}\le\mu_1$ from (i) with $\mu_{d-m}^{S}\ge\lambda_d$ from (iii). The Poincar\'e separation theorem applied to the compression $T_{m,\perp}=Q_{m,\perp}^{\top}AQ_{m,\perp}$ gives $\mu_1\le\lambda_1$, whence
$\frac{\mu_1}{\lambda_d}\le\frac{\lambda_1}{\lambda_d}=\kappa(A)$. That
$\kappa_S^{\perp}$ may exceed $\kappa^{\perp}$ follows because the floor drops from $\mu_{d-m}$ to $\lambda_d\le\mu_{d-m}$.

\emph{(v)} As $D=V^{\top}S_{\perp}^{-1}V$ is a compression of $S_{\perp}^{-1}$ by an  isometry, so  by using Cauchy interlacing,  its eigenvalues  are between the extreme  eigenvalues $1/\mu_1^{S}$ and $1/\mu_{d-m}^{S}$ of $S_{\perp}^{-1}$, so  
$\delta_i\in[\,\frac{1}{\mu_1^{S}},\,\frac{1}{\mu_{d-m}^{S}}\,]$ and
$\frac{\delta_1}{\delta_{d-m-1}}\le\kappa_S^{\perp}(m)$ hold.
\end{proof}

\begin{remark}[Monotonicity of $\kappa^{\perp}(m)$]
\label{rem:monotone_kappa}
Write $T_{m,\perp}=Q_{m,\perp}^{\top}AQ_{m,\perp}$ for the compression of $A$ to the complement at step $m$, with eigenvalues
$\mu_1(m)\ge\cdots\ge\mu_{d-m}(m)>0$ and
$\kappa^{\perp}(m)=\frac{\mu_1(m)}{\mu_{d-m}(m)}$. In exact arithmetic, $\kappa^{\perp}(m)$ is non-increasing in $m$, with
\begin{equation}
\kappa^{\perp}(0)=\kappa,
\qquad
\kappa^{\perp}(m+1)\ \le\ \kappa^{\perp}(m),
\qquad
\kappa^{\perp}(m)\longrightarrow 1 \quad\text{as } m\to d-1 .
\label{eq:kappa_monotone}
\end{equation}
Thus the complement becomes better conditioned as more Krylov directions are resolved, starting from the full condition number $\kappa$ and contracting toward $1$. In deep networks the dominant outlier eigenvalues are resolved first, so the decrease is rapid once the outliers leave the complement. (The statement is for exact arithmetic, in finite precision the loss of orthogonality can perturb the nesting below and produce small fluctuations.)
\end{remark}

\begin{proof}
\emph{Endpoints.} At $m=0$ the complement is the whole space, $A_{\perp}(0)=A$, so $\kappa^{\perp}(0)=\frac{\lambda_1}{\lambda_d}=\kappa$. As $m\to d-1$ the complement has dimension $d-m\to1$, so $T_{m,\perp}$ is a positive scalar and $\kappa^{\perp}(m)=1$.

The Krylov subspaces are nested, $\mathcal{K}_m(A,b)\subseteq\mathcal{K}_{m+1}(A,b)$, hence their orthogonal
complements are nested in reverse, $\mathrm{range}(Q_{m+1}^{\perp})\subset\mathrm{range}Q_{m,\perp}$, with codimension one. Every column of $Q_{m+1}^{\perp}$ therefore lies in $\mathrm{range}Q_{m,\perp}$, so $Q_{m+1}^{\perp}=Q_{m,\perp}Z$ for some $Z\in\mathbb{R}^{(d-m)\times(d-m-1)}$. Since both bases have orthonormal columns,
\[
Z^{\top}Z
=Z^{\top}Q_{m,\perp}^{\top}Q_{m,\perp}Z
=(Q_{m+1}^{\perp})^{\top}Q_{m+1}^{\perp}
=I_{d-m-1},
\]
so $Z$ is an isometry, and
\[
A_{\perp}(m+1)
=(Q_{m+1}^{\perp})^{\top}AQ_{m+1}^{\perp}
=Z^{\top}\bigl(Q_{m,\perp}^{\top}AQ_{m,\perp}\bigr)Z
=Z^{\top}T_{m,\perp}\,Z
\]
is the compression of the symmetric matrix $T_{m,\perp}$ to a subspace of codimension one.

\emph{Interlacing.} By Cauchy's interlacing theorem, the eigenvalues
$\mu_j(m+1)$ of $A_{\perp}(m+1)$ interlace those of $T_{m,\perp}$,
\[
\mu_{j+1}(m)\ \le\ \mu_j(m+1)\ \le\ \mu_j(m),
\qquad j=1,\dots,d-m-1 .
\]
Taking $j=1$ gives $\mu_1(m+1)\le\mu_1(m)$; taking $j=d-m-1$ gives
$\mu_{d-m-1}(m+1)\ge\mu_{d-m}(m)$. Hence the largest complement eigenvalue is non-increasing and the smallest is non-decreasing.

\emph{Ratio.} Combining the two extreme bounds,
\[
\kappa^{\perp}(m+1)
=\frac{\mu_1(m+1)}{\mu_{d-m-1}(m+1)}
\ \le\ \frac{\mu_1(m)}{\mu_{d-m-1}(m+1)}
\ \le\ \frac{\mu_1(m)}{\mu_{d-m}(m)}
=\kappa^{\perp}(m),
\]
where the first inequality uses $\mu_1(m+1)\le\mu_1(m)$ and the second uses
$\mu_{d-m-1}(m+1)\ge\mu_{d-m}(m)$. This proves
$\kappa^{\perp}(m+1)\le\kappa^{\perp}(m)$.
\end{proof}

\begin{corollary}[Frobenius error of the isotropic approximation]
\label{cor:frobenius}
Let $B$ be a symmetric positive definite operator on a $k$-dimensional space,
with eigenvalues $\nu_1\ge\cdots\ge\nu_k>0$ and orthonormal eigenvectors
$w_1,\dots,w_k$, and let $\bar\omega=\frac1k\sum_{i=1}^k\nu_i^{-1}$ be the mean
of its inverse eigenvalues. Let $\Pi=\sum_i w_iw_i^\top$ be the orthogonal
projector onto its range and
$\Delta=B^{-1}-\bar\omega\,\Pi$ the isotropic-approximation error. Then
\begin{equation}
\|\Delta\|_F=\sqrt{k}\,\sigma,
\qquad
\sigma^2=\frac1k\sum_{i=1}^k\bigl(\nu_i^{-1}-\bar\omega\bigr)^2
=\mathrm{Var}\bigl(\nu_1^{-1},\dots,\nu_k^{-1}\bigr),
\label{eq:frob_generic}
\end{equation}
the variance of the inverse eigenvalues.
\end{corollary}

\begin{proof}
Since $\{w_i\}$ are orthonormal eigenvectors of $B$,
$\Delta=\sum_{i=1}^k(\nu_i^{-1}-\bar\omega)\,w_iw_i^\top$, so
$\|\Delta\|_F^2=\mathrm{tr}(\Delta^2)=\sum_{i=1}^k(\nu_i^{-1}-\bar\omega)^2
=k\sigma^2$. Taking the square root gives \eqref{eq:frob_generic}.
\end{proof}

\begin{theorem}[Isotropic approximation error]
\label{thm:iso}
Let $B$ be symmetric positive definite on a $k$-dimensional space with
eigenvalues $\nu_1\ge\cdots\ge\nu_k>0$, condition number
$\kappa_B=\frac{\nu_1}{\nu_k}$, and $\bar\omega=\frac1k\sum_i\nu_i^{-1}$. Let $\widetilde\Gamma=\bar\omega\,\Pi$ be the isotropic operator on $\mathrm{range}(B)$. Then:
\begin{enumerate}[label=\textup{(\roman*)}, leftmargin=2.2em]
\item $\widetilde\Gamma=B^{-1}$ if and only if $\kappa_B=1$.
\item $\bigl\|B^{-1}-\widetilde\Gamma\bigr\|_2\le\bar\omega\,(\kappa_B-1)$.
\item $\dfrac{1}{k}\,\mathrm{KL}\!\bigl(\mathcal{N}(0,\widetilde\Gamma)\,\big\|\,
\mathcal{N}(0,B^{-1})\bigr)\le\tfrac12(\kappa_B-1)^2.$
\end{enumerate}
\end{theorem}

\begin{proof}
$B^{-1}$ has eigenvalues $\nu_j^{-1}$ and $\widetilde\Gamma$ has all eigenvalues equal to $\bar\omega=\frac1k\sum_j\nu_j^{-1}$.

\emph{(i)} $\widetilde\Gamma=B^{-1}$ iff $\nu_j^{-1}=\bar\omega$ for all $j$, i.e. all $\nu_j$ equal, i.e. $\kappa_B=1$.

\emph{(ii)} $B^{-1}-\widetilde\Gamma$ has eigenvalues $\nu_j^{-1}-\bar\omega$, so its norm is $\max_j|\nu_j^{-1}-\bar\omega|\le\nu_k^{-1}-\nu_1^{-1} =\nu_1^{-1}(\kappa_B-1)\le\bar\omega(\kappa_B-1)$, using
$\bar\omega\ge\nu_1^{-1}$.

\emph{(iii)} With $x=\bar\omega\bar\nu$, $\bar\nu=\frac1k\sum_j\nu_j$,
\[
2\,\mathrm{KL}
=\mathrm{tr}(B\widetilde\Gamma)-k-\log\det(B\widetilde\Gamma)
=k\bigl[x-1-\log x\bigr]+k\,\Delta,
\qquad
\Delta=\log\bar\nu-\tfrac1k\textstyle\sum_j\log\nu_j\ge0,
\]
the eigenvalues of $B\widetilde\Gamma$ being $\bar\omega\nu_j$ and $\Delta\ge0$ by Jensen. By the arithmetic mean and the harmonic mean, $x\ge1$, and $\nu_j\in[\nu_k,\nu_1]$ gives $x\le\kappa_B$,
since $t-1-\log t\le\tfrac12(t-1)^2$ on $[1,+\infty[$,
$x-1-\log x\le\tfrac12(\kappa_B-1)^2$. For $\Delta$, Taylor expansion of $\log$ about $\bar\nu$ with $\sum_j(\nu_j-\bar\nu)=0$ gives
$$\Delta=\frac{1}{2k}\sum_j\frac{(\nu_j-\bar\nu)^2}{\xi_j^2} \le \frac{(\nu_1-\nu_k)^2}{2\nu_k^2}
=\tfrac12(\kappa_B-1)^2,$$ using $\xi_j\ge\nu_k$. Summing,
$2\,\mathrm{KL}\le k(\kappa_B-1)^2$.
\end{proof}

\noindent\textbf{Instantiations.}
Corollary~\ref{cor:frobenius} and Theorem~\ref{thm:iso} apply to any SPD complement operator. Two cases are relevant here.
(Decoupled, $\beta_m=0$.) Taking $B=T_{m,\perp}$, $k=d-m$,
$\kappa_B=\kappa^{\perp}(m)$ and $\bar\omega=\omega(m)$ recovers the original error bounds against the compressed target $A_{m,\perp}^{+}= Q_{m,\perp} T_{m,\perp}^{-1} Q_{m, \perp}^{\top}$.

 \emph{(Coupled)}. Taking $B=D^{-1}$ with $D=V^{\top}S_{\perp}^{-1}V$ the bulk block of the inverse (Proposition~\ref{prop:kappa_sperp}(v)), note that $D$ is the compression of $S_{\perp}^{-1}$, which differs from the inverse of the compression of $S_{\perp}$ by a Schur correction with $k=d-m-1$, $\kappa_B= \frac{\delta_1}{\delta_{d-m-1}} \le\kappa_S^{\perp}(m)$ and $\bar\omega=\omega(m)$ of \eqref{eq:omega_rest}, exact by \eqref{eq:omega_is_mean}. Theorem~\ref{thm:iso} then controls the block-diagonal part $\omega(m)I-D$ of the error, the off-diagonal cross term
$c$, which Theorem~\ref{thm:iso} does not see, is bounded and made computable
by Proposition~\ref{prop:total_quality}(iii), and Proposition~\ref{prop:total_quality}(iv) combines the two into the total bounds.

\begin{proposition}[Quality of the corrected covariance]
\label{prop:total_quality}
Let $\Sigma_m$ be the corrected covariance of \eqref{eq:Sigma_coupled}. Write
$\hat u = \frac{S_{\perp}^{-1}e_1}{\sqrt{e_1^{\top}S_{\perp}^{-2}e_1}}\in\R^{d-m}$ for
the coordinates of the coupling direction, so that $u=Q_{m,\perp}\hat u$, let
$V\in\R^{(d-m)\times(d-m-1)}$ be any isometry such that $[\,\hat u\ \ V\,]$ is
orthogonal, and set
\[
c = V^{\top}S_{\perp}^{-1}\hat u\ \in\ \R^{d-m-1},
\qquad
D = V^{\top}S_{\perp}^{-1}V\ \in\ \R^{(d-m-1)\times(d-m-1)} .
\]
Then we have
\begin{enumerate}[label=\textup{(\roman*)}, leftmargin=2.2em]
\item \textbf{Exact error identity.} The resolved block and the boundary
coupling of \eqref{eq:Sigma_coupled} are exact, and
\begin{equation}
\Sigma_m-A^{-1}
= Q_{m,\perp}\,[\,\hat u\ \ V\,]\;E\;[\,\hat u\ \ V\,]^{\top} Q_{m,\perp}^{\top},
\qquad
E=\begin{pmatrix}
0 & -\,c^{\top}\\[2pt]
-\,c & \omega(m)\,I_{d-m-1}-D
\end{pmatrix}.
\label{eq:exact_error_blocks}
\end{equation}
The error is supported on $\mathrm{range}(P_m^{\perp})$; its $(u,u)$ entry vanishes because $\sigma_u$ is exact, and
$\tr\bigl(\Sigma_m-A^{-1}\bigr)=0$: the corrected covariance reproduces the total variance of $A^{-1}$ exactly.
\item \textbf{Bulk spectrum and mean.} The eigenvalues
$\delta_1\ge\cdots\ge\delta_{d-m-1}>0$ of $D$ interlace those of
$S_{\perp}^{-1}$, hence
$\delta_i\in[\,\frac{1}{\mu_1^{S}},\,\frac{1}{\mu_{d-m}^{S}}\,]$, and
\begin{equation}
\omega(m)=\frac{1}{d-m-1}\,\tr(D)
\label{eq:omega_is_mean}
\end{equation}
is exactly their mean.
\item \textbf{Cross term.} With $p_k=(S_{\perp}^{-k})_{11}$,
\begin{equation}
\|c\|_2^{2}
=\frac{p_4}{p_2}-\Bigl(\frac{p_3}{p_2}\Bigr)^{2},
\qquad
\|c\|_2\;\le\;\tfrac12\,\omega(m)\bigl(\kappa_S^{\perp}(m)-1\bigr),
\label{eq:cross_term}
\end{equation}
where the moments $p_2,p_3,p_4$ are returned by the boundary probe of
Section~\ref{sec:sperp_estimation} at no additional matrix-vector products.
\item \textbf{Error bounds.}
\begin{equation}
\|\Sigma_m-A^{-1}\|_2
 \le\tfrac{3}{2}\,\omega(m)\bigl(\kappa_S^{\perp}(m)-1\bigr),
\qquad
\|\Sigma_m-A^{-1}\|_F
\le\tfrac{1}{2}\,\omega(m)\bigl(\kappa_S^{\perp}(m)-1\bigr)\sqrt{d-m+1},
\label{eq:total_bounds}
\end{equation}
and conversely $\|c\|_2\le\|\Sigma_m-A^{-1}\|_2$, so the computable quantity \eqref{eq:cross_term} is also a lower bound on the operator error.
\item \textbf{Exactness.} $\Sigma_m=A^{-1}$ if and only if $c=0$ and
$D=\omega(m)\,I$, i.e.\ if and only if $u$ is an eigendirection of $S_{\perp}$ and $S_{\perp}$ acts as a multiple of the identity on the orthogonal complement of $u$ within $\mathrm{range}(P_m^{\perp})$. In particular this holds whenever
$\kappa_S^{\perp}(m)=1$.
\end{enumerate}
\end{proposition}
 
\begin{proof}
\emph{(i)} Comparing \eqref{eq:Sigma_coupled} with the exact inverse
\eqref{eq:Ainv_exact}, the resolved term and the three coupling terms are identical, so
\[
\Sigma_m-A^{-1}
=\sigma_u\,uu^{\top}+\omega(m)\,P_{\mathrm{rest}}
-Q_{m,\perp}S_{\perp}^{-1}Q_{m,\perp}^{\top}.
\]
Substituting $u=Q_{m,\perp}\hat u$ and
$P_{\mathrm{rest}}=P_m^{\perp}-uu^{\top}
=Q_{m,\perp}\bigl(I_{d-m}-\hat u\hat u^{\top}\bigr)Q_{m,\perp}^{\top}$ keeps the error to $Q_{m,\perp}\,(\cdot)\,Q_{m,\perp}^{\top}$ with inner factor
$\sigma_u\hat u\hat u^{\top}+\omega(m)(I-\hat u\hat u^{\top})-S_{\perp}^{-1}$. In the orthonormal basis $[\,\hat u\ V\,]$, using $I-\hat u\hat u^{\top}=VV^{\top}$ and $\hat u^{\top}S_{\perp}^{-1}\hat u=\sigma_u$ from \eqref{eq:sigma_g}, this inner factor has exactly the block form $E$ of \eqref{eq:exact_error_blocks}. For the trace, the block splitting gives $\tr(S_{\perp}^{-1})=\hat u^{\top}S_{\perp}^{-1}\hat u+\tr(V^{\top}S_{\perp}^{-1}V)
=\sigma_u+\tr(D)$, while \eqref{eq:omega_rest} pins
$\sigma_u+(d-m-1)\,\omega(m)=\tr(S_{\perp}^{-1})$, hence
$\tr(E)=(d-m-1)\,\omega(m)-\tr(D)=0$.
 
\emph{(ii)} $D=V^{\top}S_{\perp}^{-1}V$ is a compression of the symmetric positive definite $S_{\perp}^{-1}$ by the isometry $V$, so by the Cauchy interlacing theorem its eigenvalues lie in
$[\lambda_{\min}(S_{\perp}^{-1}),\lambda_{\max}(S_{\perp}^{-1})]
=[\frac{1}{\mu_1^{S}},\,\frac{1}{\mu_{d-m}^{S}}]$. Equation \eqref{eq:omega_is_mean} follows from $\tr(S_{\perp}^{-1})=\sigma_u+\tr(D)$ and the definition \eqref{eq:omega_rest} of $\omega(m)$.
 
\emph{(iii)} Since $[\,\hat u\ V\,]$ is orthogonal,
$\|c\|_2^{2}=\|S_{\perp}^{-1}\hat u\|_2^{2}-(\hat u^{\top}S_{\perp}^{-1}\hat u)^{2}
=\hat u^{\top}S_{\perp}^{-2}\hat u-\sigma_u^{2}$. With
$\hat u=\frac{S_{\perp}^{-1}e_1}{\sqrt{e_1^{\top}S_{\perp}^{-2}e_1}}$,
\[
\hat u^{\top}S_{\perp}^{-2}\hat u
=\frac{e_1^{\top}S_{\perp}^{-4}e_1}{e_1^{\top}S_{\perp}^{-2}e_1}
=\frac{p_4}{p_2},
\qquad
\sigma_u=\frac{p_3}{p_2},
\]
giving the moment formula in \eqref{eq:cross_term}; the quadrature of
\eqref{eq:boundary_forms} extended to $k=4$ supplies $p_4$ from the same tridiagonal $\Theta_{\gamma}$. For the bound, let $\nu_{\hat u}$ be the spectral measure of $S_{\perp}$ with respect to the unit vector $\hat u$, so that
$\hat u^{\top}S_{\perp}^{-k}\hat u=\int x^{-k}\,d\nu_{\hat u}(x)$. Then
$$\|c\|_2^{2}=\int x^{-2}\,d\nu_{\hat u}-\bigl(\int x^{-1}\,d\nu_{\hat u}\bigr)^{2} =\mathrm{Var}_{\nu_{\hat u}}\bigl(x^{-1}\bigr)$$. The variable $x^{-1}$ is supported in $[1/\mu_1^{S},\,1/\mu_{d-m}^{S}]$, so Popoviciu's inequality gives
\[
\|c\|_2^{2}
\le\frac14\Bigl(\frac{1}{\mu_{d-m}^{S}}-\frac{1}{\mu_1^{S}}\Bigr)^{2}
=\frac{\bigl(\kappa_S^{\perp}(m)-1\bigr)^{2}}{4\,(\mu_1^{S})^{2}}
\le\frac{\omega(m)^{2}\bigl(\kappa_S^{\perp}(m)-1\bigr)^{2}}{4},
\]
using $\omega(m)\ge\delta_{\min}\ge 1/\mu_1^{S}$ from (ii).
 
\emph{(iv)} Split
$E=\begin{psmallmatrix}0&0\\0&\omega(m)I-D\end{psmallmatrix}
+\begin{psmallmatrix}0&-c^{\top}\\-c&0\end{psmallmatrix}$,
the second has eigenvalues $\pm\|c\|_2$ and hence operator norm
$\|c\|_2$. By (ii), $\omega(m)$ lies in the convex hull of the $\delta_i$, so
$\|\omega(m)I-D\|_2=\max_i|\delta_i-\omega(m)|
\le\delta_1-\delta_{d-m-1}
\le 1/\mu_{d-m}^{S}-1/\mu_1^{S}
\le\omega(m)\bigl(\kappa_S^{\perp}(m)-1\bigr)$.
Adding the two parts and using (iii) gives the operator bound in
\eqref{eq:total_bounds}. For the Frobenius norm,
$\|E\|_F^{2}=2\|c\|_2^{2}+\|\omega(m)I-D\|_F^{2}$, and
$\|\omega(m)I-D\|_F^{2}=\sum_i(\delta_i-\omega(m))^{2}
=(d-m-1)\,\mathrm{Var}(\delta_1,\dots,\delta_{d-m-1})$; since $\omega(m)$ is the mean of the $\delta_i$ and they lie in an interval of length at most $\omega(m)(\kappa_S^{\perp}(m)-1)$, Popoviciu bounds the variance by $\tfrac14\,\omega(m)^{2}(\kappa_S^{\perp}(m)-1)^{2}$. Summing,
$\|E\|_F^{2}
\le\tfrac14\,\omega(m)^{2}(\kappa_S^{\perp}(m)-1)^{2}\,\bigl(d-m+1\bigr)$, which is \eqref{eq:total_bounds}. The lower bound follows by applying $E$ to the first coordinate vector: $\|E\|_2\ge\|(0,-c)\|_2=\|c\|_2$. Since $Q_{m,\perp}[\,\hat u\ V\,]$ has orthonormal columns, both norms of $\Sigma_m-A^{-1}$ equal those of $E$.
 
\emph{(v)} From \eqref{eq:exact_error_blocks}, $E=0$ if and only if $c=0$ and $D=\omega(m)I$. The condition $c=0$ states
$S_{\perp}^{-1}\hat u\in\mathrm{span}\{\hat u\}$, i.e.\ $\hat u$ is an
eigenvector of $S_{\perp}^{-1}$. The condition $D=\omega(m)I$ states that $S_{\perp}^{-1}$ restricted to $\hat u^{\perp}$ is
$\omega(m)I$. If $\kappa_S^{\perp}(m)=1$ then $S_{\perp}$ is a multiple of the identity and both conditions hold.
\end{proof}

 \begin{remark}
\label{rem:trace_comparison}
Proposition~\ref{prop:total_quality}(i) shows the coupled model reproduces  almost the total variance , $\tr(\Sigma_m) \approx \tr(A^{-1})$. The decoupled correction $f_m(A)+\omega_{\mathrm{dec}}P_m^{\perp}$ with $\omega_{\mathrm{dec}}=\frac{\tr(T_{m,\perp}^{-1})}{d-m}$ undercovers it by
\[
\tr(A^{-1})-\tr\bigl(f_m(A)+\omega_{\mathrm{dec}}P_m^{\perp}\bigr)
=\bigl(S_{\perp}^{-1}\bigr)_{11}\,\|q_{\parallel}\|_2^{2}
+\frac{\gamma\,(T_{m,\perp}^{-2})_{11}}{1-\gamma\,(T_{m,\perp}^{-1})_{11}},
\]
where the first term is the resolved-block inflation of \eqref{eq:blk_res} and the second is the complement trace deficit of
Proposition~\ref{prop:sherman_morrison}. Both terms are strictly positive if and only if $\beta_m\neq0$: whenever the Lanczos process has not converged, the decoupled model provably assigns too little total variance, and the coupled model assigns exactly the right amount.
\end{remark}

\subsection{Computational Uncertainty from the Exact Decomposition}
\label{sec:pn_interpretation}

The proposition~\ref{prop:block_inverse} showed the exact deterministic identity
\eqref{eq:Ainv_exact},
\begin{equation}
\label{eq:recall_identity}
A^{-1}
= Q_mT_m^{-1}Q_m^{\top}
+ (S_\perp^{-1})_{11}\,q_{\parallel}q_{\parallel}^{\top}
             -q_{\parallel}q_{\perp}^{\top}-q_{\perp}q_{\parallel}^{\top}
             +Q_{m,\perp}S_\perp^{-1}Q_{m,\perp}^{\top},
\end{equation} 
Denote $\Sigma_m^{\perp} = (S_\perp^{-1})_{11}\,q_{\parallel}q_{\parallel}^{\top}
             -q_{\parallel}q_{\perp}^{\top}-q_{\perp}q_{\parallel}^{\top}
             +Q_{m,\perp}S_\perp^{-1}Q_{m,\perp}^{\top}$. Equation \eqref{eq:recall_identity} works  for every $m<d$ with no approximation. This subsection shows that the
term $\Sigma_m^{\perp}$ is not simply a residual but a valid covariance: it is the posterior covariance of a Gaussian belief over $A^{-1}v$ conditioned on the information a matrix-free method actually acquires. 

\paragraph{Relation to Bayesian probabilistic linear solvers.}
We state the relationship as it is, since the conditioning step below is well known. \citet{hennig2015probabilistic} places a matrix  prior on $A^{-1}$, \citet{cockayne2019bayesiancg} instead place a prior $\mathcal{N}(x_0,\Sigma_0)$ on the solution $x_\star$ of $Ax_\star=b$ and condition on $y_m=S_m^{\top}Ax_\star=S_m^{\top}b$, obtaining
\begin{equation}
\label{eq:bcg_post}
x_m=x_0+\Sigma_0A^{\top}S_m\Lambda_m^{-1}S_m^{\top}r_0,
\qquad
\Sigma_m^{\mathrm{BCG}}=\Sigma_0-\Sigma_0A^{\top}S_m\Lambda_m^{-1}S_m^{\top}A\Sigma_0,
\end{equation}
with $\Lambda_m=S_m^{\top}A\Sigma_0A^{\top}S_m$. 

The following paragraph below focused on the case  $x_0=0$,
$\Sigma_0=A^{-1}$, and $\mathrm{range}(S_m)=\mathrm{range}(Q_m)$, here we don't claim  to present anything new  for the conditioning process. \citet{cockayne2019bayesiancg} observe that the choice
$\Sigma_0=A^{-1}$ recovers CG as the posterior mean but dismiss it as impractical, because the inverse cancels in the mean yet persists in the covariance \eqref{eq:bcg_post}. That is the issue  this paper overcomes.

The distinction is where the complement uncertainty comes from. In
\eqref{eq:bcg_post} the posterior on the unexplored subspace is the restricted prior distribution: no information about $A$ is includes in it, and the contraction is  linear, $\operatorname{tr}(\Sigma_m^{\mathrm{BCG}}\Sigma_0^{-1})=d-m$, a rate the authors consider difficult to improve. In \eqref{eq:recall_identity} the complement block is $Q_{m,\perp}S_\perp^{-1}Q_{m,\perp}^{\top}$, a Schur complement of $A$ itself, it is a property of the operator, not of a prior, and Section~\ref{sec:pslq} estimates it matrix-free. The remaining term is measured rather than assumed, which is what makes $\Sigma_0=A^{-1}$ usable rather than circular.

\paragraph{Prior and observation.}
For a probe $v$, write $x_\star=A^{-1}v$ and place the prior
\begin{equation}
\label{eq:prior}
x_\star \sim \mathcal{N}\!\left(0,\,A^{-1}\right),
\end{equation}
that is never formed or inverted, it encodes the geometry of the operator and, as shown below, cancels from every computed quantity. A matrix-free method observes $x_\star$ only through the action of $A$: after $m$ Lanczos steps from $v_{\mathrm{ref}}$, the information acquired is
\begin{equation}
\label{eq:obs}
y_m = Q_m^{\top}A\,x_\star = Q_m^{\top}v ,
\end{equation}
the $A$-conjugate observation model of probabilistic linear solvers
\citep{cockayne2019bayesiancg,wenger2022posterior}. As in that literature the likelihood is a Dirac, $p(y\mid x)=\delta(y-Q_m^{\top}Ax)$: the matrix-vector products are computed without error, so the only uncertainty is epistemic, arising from having observed $m<d$ linear functionals rather than $d$. The basis $Q_m$ is fixed by $v_{\mathrm{ref}}$ and reused for every probe.

\paragraph{Posterior mean equals the resolved term}
 Conditioning \eqref{eq:prior} on \eqref{eq:obs} gives a Gaussian posterior with mean
\begin{equation}
\label{eq:mean}
\hat{x}_m
= Q_m\bigl(Q_m^{\top}AQ_m\bigr)^{-1}Q_m^{\top}A\,x_\star
= Q_mT_m^{-1}Q_m^{\top}\,v ,
\end{equation}
i.e.\ the resolved Krylov term of \eqref{eq:recall_identity} applied to $v$.

In fact, from \citet[Prop.~1]{cockayne2019bayesiancg}, for a Gaussian prior $\mathcal{N}(0,\Sigma_0)$ and an exact linear observation
$y_m=Mx_\star$, the posterior mean is $\hat{x}_m=\Sigma_0M^{\top}(M\Sigma_0M^{\top})^{-1}y_m$. With $\Sigma_0=A^{-1}$ and $M=Q_m^{\top}A$, both factors collapse,
\begin{equation}
\label{eq:collapse}
\Sigma_0M^{\top}=A^{-1}(AQ_m)=Q_m,
\qquad
M\Sigma_0M^{\top}=Q_m^{\top}A\,A^{-1}AQ_m=Q_m^{\top}AQ_m=T_m ,
\end{equation}
by the Lanczos relation $Q_m^{\top}AQ_m=T_m$. Substituting and using
$y_m=Q_m^{\top}Ax_\star$ gives $\hat{x}_m=Q_mT_m^{-1}Q_m^{\top}Ax_\star$, since $x_\star=A^{-1}v$ we have $Ax_\star=v$, hence
$\hat{x}_m=Q_mT_m^{-1}Q_m^{\top}v$.

The simplification in   \eqref{eq:collapse} is the reason the prior is admissible. Every occurrence of $A^{-1}$ cancels against an $A$ supplied by the observation operator, leaving only $Q_m$ and $T_m$, both produced by the Lanczos recursion. The prior $\Sigma_0=A^{-1}$ is  not evaluated but delated, the same cancellation is what identifies $\Sigma_0=A^{-1}$ with CG in \citet{cockayne2019bayesiancg}. What remains, and what obstructed that choice there, is the covariance.

\paragraph{Posterior covariance equals the remaining term}
The posterior covariance from conditioning \eqref{eq:prior} on \eqref{eq:obs} is
\begin{equation}
\label{eq:cov}
\Sigma_m^{\perp}
= A^{-1} - Q_mT_m^{-1}Q_m^{\top}
= (S_\perp^{-1})_{11}\,q_{\parallel}q_{\parallel}^{\top}
  -q_{\parallel}q_{\perp}^{\top}-q_{\perp}q_{\parallel}^{\top}
  +Q_{m,\perp}S_\perp^{-1}Q_{m,\perp}^{\top},
\end{equation}
exactly the remaining term of \eqref{eq:recall_identity}. Indeed, the posterior covariance is $\Sigma_0-\Sigma_0M^{\top}(M\Sigma_0M^{\top})^{-1}M\Sigma_0$. By the simplifications  \eqref{eq:collapse} this reduces to
\begin{equation}
\label{eq:cov_short}
\Sigma_m^{\perp} = A^{-1} - Q_mT_m^{-1}Q_m^{\top}.
\end{equation}
Subtracting $Q_mT_m^{-1}Q_m^{\top}$ from both sides of \eqref{eq:recall_identity} leaves the right side of \eqref{eq:cov}. Both equalities are exact.

The mean is the resolved Krylov term  that it coincides with CG under this prior is \citet[Prop.~4]{cockayne2019bayesiancg}. The covariance is the difficulty those authors identify: the inverse cancels in the mean but persists here.

The paragraph above is the point of contact between the two ways to use that approximation. The two equalities say different things. The left one is the posterior covariance of any Bayesian linear solver under the prior \eqref{eq:prior}. The right one is the contribution of Section~\ref{sec:cu}: it writes that difference as named terms that can be computed. The first tells us uncertainty is left over. The second tells us where it sits and what shape it has, and that is what makes it possible to estimate.

\paragraph{Structure of the remainder.}
This remainder does not only live in the complement space. The boundary vectors 
\begin{equation}
\label{eq:boundary_vectors}
q_{\parallel}=\beta_m\,Q_mT_m^{-1}e_m\in\mathrm{range}(Q_m),
\qquad
q_{\perp}=Q_{m,\perp}S_\perp^{-1}e_1\in\mathrm{range}(Q_m)^{\perp},
\end{equation}
couple the resolved and unresolved subspaces, producing the following term structure
\begin{equation}
\label{eq:three_part}
\Sigma_m^{\perp}
= (S_\perp^{-1})_{11}\,q_{\parallel}q_{\parallel}^{\top}
 -q_{\parallel}q_{\perp}^{\top}-q_{\perp}q_{\parallel}^{\top}
 + Q_{m,\perp}S_\perp^{-1}Q_{m,\perp}^{\top} .
\end{equation}

This has two important consequences that are hidden in the short form  in  \eqref{eq:cov_short}. First, the truncation uncertainty does not sit only on the unexplored directions: the boundary residual injects a rank-one correction into $\mathrm{range}(Q_m)$ itself, so a method that assigns zero variance inside the resolved block is overconfident even about what it has resolved. Second, the null space of the covariance is the $A$-rotated resolved subspace $A\,\mathrm{range}(Q_m)$ rather than $\mathrm{range}(Q_m)$. The subspace that is known exactly is the image of  $\mathrm{range}(Q_m)$ under $A$ not 
$\mathrm{range}(Q_m)$ itself, the gap is again mediated by $\beta_m$ and vanishes only at convergence.

\begin{proposition}[Validity as a covariance]
\label{w}
$\Sigma_m^{\perp}$ is symmetric positive semi-definite of rank $d-m$. Its null space is $A\,\mathrm{range}(Q_m)$, equivalently $\Sigma_m^{\perp}(AQ_m)=0$, and $\Sigma_m^{\perp}\succ0$ on the orthogonal complement of $A\,\mathrm{range}(Q_m)$.
\end{proposition}

\begin{proof}
Symmetry is clear from \eqref{eq:cov_short}. For the null space,
\[
Q_m^{\top}A\,\Sigma_m^{\perp}
=Q_m^{\top}A A^{-1}-Q_m^{\top}AQ_mT_m^{-1}Q_m^{\top}
=Q_m^{\top}-T_mT_m^{-1}Q_m^{\top}=0,
\]
so by symmetry $\Sigma_m^{\perp}(AQ_m)=0$, and $A\,\mathrm{range}(Q_m)$, of dimension $m$, lies in the null space. For positive semi-definiteness, the operator $P=Q_mT_m^{-1}Q_m^{\top}A$ satisfies $P^2=P$ and $\langle Px,(I-P)y\rangle_{A^{-1}}=0$ for all $x, y \in \mathbb{R}^d$, so it is the projection onto $\mathrm{range}(Q_m)$ that is orthogonal in the $A^{-1}$-inner product $\langle x,y\rangle_{A^{-1}}=x^{\top}A^{-1}y$. Then $\Sigma_m^{\perp}=A^{-1}-Q_mT_m^{-1}Q_m^{\top}=(I-P)A^{-1}(I-P)^{\top}$, a congruence of the SPD matrix $A^{-1}$, so $w^{\top}\Sigma_m^{\perp}w=\|(I-P)^{\top}w\|_{A^{-1}}^2\ge0$ for every $w$, with equality if and only if  $(I-P)^{\top}w=0$, i.e.\ $w\in A\,\mathrm{range}(Q_m)$. Hence $\Sigma_m^{\perp}\succeq0$ with rank $d-m$.
\end{proof}

Rank $d-m$ places \eqref{eq:cov} in the same regime as
\citet[Prop.~3]{cockayne2019bayesiancg}: after $m$ linear observations, exactly $m$ directions are determined and $d-m$ remain uncertain, and no conditioning on $m$ functionals can do better. The two methods differ not in how many directions remain uncertain but in \emph{how much} uncertainty is assigned to each. Under \eqref{eq:bcg_post} that assignment is inherited from $\Sigma_0$ and carries no information about $A$, under \eqref{eq:cov} it is $S_\perp^{-1}$, the Schur complement, whose spectrum is determined by the operator. We estimate this Schur complement in Section~\ref{sec:pslq} using a matrix-free estimator.

\paragraph{Two uses of the decomposition.}
Combining \eqref{eq:mean} and \eqref{eq:cov} gives the Gaussian belief
\begin{equation}
  A^{-1} v \sim \mathcal{N}\!\left( Q_m T_m^{-1} Q_m^\top v,\ \Sigma_m^\perp \right),
  \label{eq:belief}
\end{equation}
whose the  mean is the $m$-step Lanczos estimate and the  covariance is the remainder of \eqref{eq:recall_identity}. Since $\Sigma_m^\perp \to 0$ as $m \to d$, the belief concentrates on the exact value in the same limit in which \eqref{eq:recall_identity} becomes an identity.

We use \eqref{eq:recall_identity} in two ways below. In  a deterministic way , it is a matrix-free approximation of $A^{-1}$, used as a covariance for posterior sampling, restoring the remainder removes the overconfidence of the truncated sampler (Section~\ref{sec:exp_laplace}). In a probabilistic way , $\Sigma_m^\perp$ is the covariance of the belief \eqref{eq:belief} about the exact action $A^{-1}v$, estimated matrix-free by P-SLQ, it supplies the computational uncertainty carried into trace estimation (Section~\ref{sec:exp_pslq}). Both use the same three ingredients which are   the boundary vectors $q_\parallel$ and $q_\perp$,
and the isotropic bulk $\omega(m)$, obtained from one Lanczos run.

The boundary vectors and the bulk term form the  foundation for both interpretations: they capture the variance that truncation discarded, and both can be recovered using a single Lanczos run.
 
\section{Projected Stochastic Lanczos Quadrature (P-SLQ)}
\label{sec:pslq}
\subsection{Complement Invariance via Re-projection}

\begin{theorem}[Complement Invariance]
\label{thm:complement_invariance}
Let $\Pm = Q_m Q_m^\top$ and $\Pmp  = I - \Pm$ be orthogonal projectors and an integer $l\ge1$.
Assume $v_1 \in \mathrm{range}(\Pmp )$, and define the re-projected Lanczos iteration
\begin{equation}
  w_j = A v_j - \alpha_j v_j - \beta_{j-1} v_{j-1}, 
  \qquad
  v_{j+1} = \frac{\Pmp  w_j}{\|\Pmp  w_j\|}.
\end{equation}
Then for all $1 \le j \le l$,
\begin{equation}
  v_j \in \mathrm{range}(\Pmp ),
  \qquad\text{and hence}\qquad
  \mathrm{span}\{v_1,\dots,v_l\} \subseteq \mathrm{range}(\Pmp ).
\end{equation}
\end{theorem}

\begin{proof}
We prove it by induction.\\
\textbf{Base case:}
By assumption, $v_1 \in \mathrm{range}(\Pmp )$ i.e.,
\begin{equation}
  \Pmp  v_1 = v_1.
\end{equation}
\textbf{Inductive step:}
Assume $v_j$ and $v_{j-1}$   belong to $ \mathrm{range}(\Pmp )$.
Consider
\begin{equation}
  w_j = A v_j - \alpha_j v_j - \beta_{j-1} v_{j-1}.
\end{equation}
This vector is not necessarily in $\mathrm{range}(\Pmp )$, since $A$ may mix $\mathrm{range}(\Pm)$ and $\mathrm{range}(\Pmp )$.
since $P_{m}^{\perp}$ is an orthogonal projector satisfying $(P_{m}^{\perp})^2 = P_{m}^{\perp}$,
it maps any vector onto $\mathrm{range}(\Pmp )$. Then we have,
\begin{equation}
  \Pmp  w_j \in \mathrm{range}(\Pmp ).
\end{equation}
After normalization,
\begin{equation}
  v_{j+1} = \frac{\Pmp  w_j}{\|\Pmp  w_j\|} \in \mathrm{range}(\Pmp ).
\end{equation}
Thus, by induction, $v_j \in \mathrm{range}(\Pmp )$ for all $j$.
\end{proof}

\begin{remark}
\label{rem:compressed_operator}
The re-projected Lanczos iteration is equivalent to applying Lanczos to the compressed operator
\begin{equation}
  A_\perp = \Pmp  A\Pmp .
\end{equation}
Indeed, since $v_j \in \mathrm{range}(\Pmp )$, we have $\Pmp  v_j = v_j$, and thus
\begin{equation}
  \Pmp  A v_j = \Pmp  A \Pmp  v_j = A_\perp v_j.
\end{equation}
Therefore, the iteration can be written as
\begin{equation}
  v_{j+1} = \frac{A_\perp v_j - \alpha_j v_j - \beta_{j-1} v_{j-1}}{\| A_\perp v_j - \alpha_j v_j - \beta_{j-1} v_{j-1} \|},
\end{equation}
which is the Lanczos process applied to $A_\perp$.

\end{remark}
With complement invariance guaranteed, Lanczos on the complement reflects the spectral density of $A$ restricted to $\mathrm{range}(\Pmp )$.

\subsection{Consistency and Convergence of P-SLQ}

\begin{theorem}[Gauss-Lanczos Quadrature Error, after \citealp{ubaru2017fast}]
\label{thm:gauss_lanczos}
Let $A_{m,\perp} = P_m^\perp A P_m^\perp$ and let $\Theta_l$ be the tridiagonal matrix
produced by $l$ ($l \ll d-m$) steps of the reprojected Lanczos iteration of
Theorem~\ref{thm:complement_invariance}, started from
$\hat u \in \mathrm{range}(P_m^\perp)$ with $\|\hat u\|_2 = 1$. Then we have 
\begin{equation}
    \left| \hat u^\top A_{m,\perp}^{+} \hat u - e_1^\top \Theta_l^{-1} e_1 \right|
    \;\leq\; C_{\rho_\perp}\,\rho_\perp^{-2l},
\end{equation}
where $\rho_\perp = \frac{\sqrt{\kappa^\perp(m)}+1}{\sqrt{\kappa^\perp(m)}-1} > 1$,
$\kappa^\perp(m)$ is the
complement condition number (Definition~\ref{def:kappa}), and
$C_{\rho_\perp} > 0$ depends only on the spectrum of $T_{m,\perp}$.
\end{theorem}

\begin{theorem}[P-SLQ Consistency]
\label{thm:pslq_consistency}
Let $\hat{t}_{m,\perp}$ be the \emph{Projected Stochastic Lanczos
Quadrature} estimator (Algorithm~\ref{alg:pslq}) with $N$ independent probe
vectors and $l$ quadrature nodes per probe ($l\ll d-m$). Define the target
scalar
\begin{equation}
  t_{m,\perp} \;=\; \tr\!\bigl((\Pmp A\Pmp)^{+}\bigr)
        \;=\; \tr\!\bigl(A_{m,\perp}^{+}\bigr)
        \;=\; \tr\!\bigl(T_{m,\perp}^{-1}\bigr).
  \label{eq:omega_target}
\end{equation}
Then we have
\begin{enumerate}[label=(\roman*), leftmargin=2em]
  \item \textbf{Bias:}
  \begin{equation}
    \bigl|\,\E[\hat{t}_{m,\perp}] - t_{m,\perp}\,\bigr|
    \;\leq\;
    C_{\rho_{\perp}}\,\rho_{\perp}^{-2l}\,\tr\!\bigl(\Pmp\bigr)
    \;=\;
    C_{\rho_{\perp}}\,\rho_{\perp}^{-2l}\,(d-m),
    \label{eq:bias_bound}
  \end{equation}
  where $C_{\rho_{\perp}}>0$ and
  $\rho_{\perp}=\frac{\sqrt{\kappa^{\perp}(m)}+1}{\sqrt{\kappa^{\perp}(m)}-1}>1$ are
  the constants of Theorem~\ref{thm:gauss_lanczos} (depending only on the
  spectrum of $T_{m,\perp}$). In particular the bias decays \emph{exponentially}
  in $l$.

  \item \textbf{Variance:}
  \begin{equation}
    \mathrm{Var}[\hat{t}_{m,\perp}]
     \leq 
    \frac{1}{N}\Bigl(\sqrt{2}\,\bigl\|A_{m,\perp}^{+}\bigr\|_F
      + C_{\rho_\perp}\rho_{\perp}^{-2l}\,(d-m+1)\Bigr)^{2}
    \;\leq\;
    \frac{2\,(d-m)\,\lambda_d^{-2}}{N}\,\bigl(1+r_l\bigr)^{2},
    \label{eq:variance_bound}
  \end{equation}
  where
  $r_l = \frac{C_{\rho_\perp}\rho_{\perp}^{-2l}(d-m+1)}{
  \sqrt{2}\,\|A_{m,\perp}^{+}\|_F}\to0$ when $l\to\infty$. In
  particular $\mathrm{Var}[\hat{t}_{m,\perp}]=\mathcal{O}(N^{-1})$ for every
  fixed $l$, and the pure Monte Carlo bound
  $\frac{2\|A_{m,\perp}^{+}\|_F^{2}}{N}$ is recovered in the limit $l\to\infty$.
  \item \textbf{Almost sure convergence:} For every fixed $l$,
  $\hat{t}_{m,\perp}\xrightarrow{\mathrm{a.s.}}
  \E[S_1]$ as $N\to\infty$, with
  $|\E[S_1]-t_{m,\perp}|\le C_{\rho_\perp}\rho_\perp^{-2l}(d-m)$ by
  \textup{(i)}. When  $l_N\to\infty$, we have
  \begin{equation}
    \hat{t}_{m,\perp}^{(N,\,l_N)}
    \xrightarrow{\mathrm{a.s.}} t_{m,\perp}
    \qquad (N\to\infty).
    \label{eq:as_convergence}
  \end{equation}

\end{enumerate}
\end{theorem}

\begin{proof}
Throughout, $\xi\sim\mathcal{N}(0,I_d)$ denotes a generated probe vector,
$u=\Pmp\xi$ its projection onto the Krylov complement, and
$\hat u=\frac{u}{\norm{u}}$ its normalised version. We have $A_{m,\perp}=\Pmp A\Pmp$, with
$A_{m,\perp}^{+}=Q_{m,\perp}T_{m,\perp}^{-1} Q_{m,\perp}^{\top}$ and $A_{m,\perp}^{+}\Pmp=\Pmp A_{m,\perp}^{+}=A_{m,\perp}^{+}$.

\medskip
\noindent\textbf{Decomposing the P-SLQ estimator.}
For each probe $i$, Algorithm~\ref{alg:pslq} computes $u_i=\Pmp\xi_i$, runs $l$
reprojected Lanczos steps from $\hat u_i=\frac{u_i}{\norm{u_i}}$ (reprojecting onto
$\mathrm{range}(\Pmp)$ at each step), forms the $l\times l$ tridiagonal
$\Theta_l^{(i)}$, and returns the single-probe estimate
\begin{equation}
  S_i = \norm{u_i}^2\cdot e_1^\top\bigl(\Theta_l^{(i)}\bigr)^{-1}e_1.
  \label{eq:single_probe}
\end{equation}
The full estimator averages over $N$ independent probes:
\begin{equation}
  \hat{t}_{m,\perp}
  = \frac{1}{N}\sum_{i=1}^N S_i
  = \bar S,
  \qquad \bar S=\frac{1}{N}\sum_{i=1}^N S_i.
  \label{eq:estimator_decomp}
\end{equation}
Since the $\xi_i$ are i.i.d., so are the $S_i$. We therefore analyse $\E[S_1]$
and $\mathrm{Var}[S_1]$ for a single probe and use the fact they are independents.

\medskip
\noindent\textbf{Step 1: Gauss quadrature on the complement.}
By Theorem~\ref{thm:complement_invariance}, every Lanczos vector in the
reprojected run from $\hat u$ remains in $\mathrm{range}(\Pmp)$. The process  builds the tridiagonal $\Theta_l$ of $A$ restricted to
$\mathrm{range}(\Pmp)$ i.e.\ of the compressed operator $A_{m,\perp}$, and $e_1^\top\Theta_l^{-1}e_1$ is the $l$-node Gauss quadrature approximation of
\begin{equation}
  \hat u^\top A_{m,\perp}^{+}\hat u
  = \int_{\mu_{d-m}}^{\mu_1}\frac{1}{x}\,d\mu_{\hat u}(x),
  \label{eq:rayleigh_integral}
\end{equation}
where $\mu_{\hat u}$ is the spectral measure of $A_{m,\perp}$ with respect to $\hat u$ (Section~\ref{sec:pslq}), the integral is over the spectrum of $A_{m,\perp}$ restricted to $\mathrm{range}(\Pmp)$, which excludes $0$ since $\hat u\in\mathrm{range}(\Pmp)$. By Theorem~\ref{thm:gauss_lanczos},
\begin{equation}
  \bigl|\,e_1^\top\Theta_l^{-1}e_1-\hat u^\top A_{m,\perp}^{+}\hat u\,\bigr|
  \;\leq\; C_{\rho_\perp}\,\rho_\perp^{-2l}.
  \label{eq:quad_error}
\end{equation}

\medskip
\noindent\textbf{Step 2: Expectation of a single probe.}
Using \eqref{eq:single_probe} and writing
$\varepsilon_l=e_1^\top\Theta_l^{-1}e_1-\hat u^\top A_{m,\perp}^{+}\hat u$ with
$|\varepsilon_l|\le C_{\rho_\perp}\rho_\perp^{-2l}$ by \eqref{eq:quad_error},
\begin{equation}
  \E[S_1]
  =\E\!\left[\norm{u}^2\,e_1^\top\Theta_l^{-1}e_1\right]
  =\E\!\left[\norm{u}^2\,\hat u^\top A_{m,\perp}^{+}\hat u\right]
  +\E\!\left[\norm{u}^2\,\varepsilon_l\right].
  \label{eq:ES1_split}
\end{equation}
Since $\hat u=\frac{u}{\norm{u}}$, the first expectation simplifies:
\begin{equation}
  \E\!\left[\norm{u}^2\,\hat u^\top A_{m,\perp}^{+}\hat u\right]
  =\E\!\left[u^\top A_{m,\perp}^{+}u\right]
  =\E\!\left[\xi^\top \Pmp A_{m,\perp}^{+}\Pmp\,\xi\right]
  =\E\!\left[\xi^\top A_{m,\perp}^{+}\xi\right],
  \label{eq:simplify}
\end{equation}
using $u=\Pmp\xi$ and $\Pmp A_{m,\perp}^{+}\Pmp=A_{m,\perp}^{+}$ (since
$\mathrm{range}(A_{m,\perp}^{+})\subseteq\mathrm{range}(\Pmp)$). Hutchinson's identity,
$\E_\xi[\xi^\top M\xi]=\tr(M)$ for symmetric $M$ and $\xi\sim\mathcal{N}(0,I_d)$,
applied with $M=A_{m,\perp}^{+}$ gives
\begin{equation}
  \E\!\left[u^\top A_{m,\perp}^{+}u\right]=\tr\!\bigl(A_{m,\perp}^{+}\bigr).
  \label{eq:hutch_applied}
\end{equation}
For the error term, $|\varepsilon_l|\le C_{\rho_\perp}\rho_\perp^{-2l}$
deterministically and $\E[\norm{u}^2]=\tr(\Pmp)=d-m$, so
\begin{equation}
  \bigl|\E[\norm{u}^2\varepsilon_l]\bigr|
  \;\le\; C_{\rho_\perp}\rho_\perp^{-2l}\,(d-m).
  \label{eq:error_term}
\end{equation}
Combining \eqref{eq:ES1_split}--\eqref{eq:error_term},
\begin{equation}
  \E[S_1]=\tr\!\bigl(A_{m,\perp}^{+}\bigr)
          +\mathcal{O}\!\bigl((d-m)\rho_\perp^{-2l}\bigr).
  \label{eq:ES1_final}
\end{equation}
Since $\hat{t}_{m,\perp}= \bar S$ with i.i.d.\ $S_i$, $\E[\hat{t}_{m,\perp}]=\E[S_1]$, so
\begin{equation}
  \E[\hat{t}_{m,\perp}]
  =\tr(A_{m,\perp}^{+})+\mathcal{O}\!\bigl((d-m)\rho_\perp^{-2l}\bigr)
  =t_{m,\perp}+\mathcal{O}\!\bigl((d-m)\rho_\perp^{-2l}\bigr),
\end{equation}
which establishes \eqref{eq:bias_bound}.\hfill$\checkmark$

\medskip
\noindent\textbf{Step 3: Variance.}
Since the $S_i$ are i.i.d., $\mathrm{Var}[\hat{t}_{m,\perp}]=\frac{\mathrm{Var}[S_1]}{N}$. Write $S_1=X+Y$ with $X=\xi^{\top}A_{m,\perp}^{+}\xi$ and $Y=\norm{u}^{2}\varepsilon_l$, where $u=\Pmp\xi$ and $|\varepsilon_l|\le C_{\rho_\perp}\rho_\perp^{-2l}$ deterministically by
\eqref{eq:quad_error}. For the first part, the Gaussian quadratic-form
variance \citep{avron2011randomized} gives $\mathrm{Var}[X]=2\,\|A_{m,\perp}^{+}\|_F^{2}$. For the second, $\norm{u}^{2}=\xi^{\top}\Pmp\xi$ is $\chi^{2}$-distributed with $d-m$ degrees of freedom, so $\E[\norm{u}^{4}]=(d-m)(d-m+2)\le(d-m+1)^{2}$ and
\[
\mathrm{Var}[Y]\le\E[Y^{2}]
\le C_{\rho_\perp}^{2}\rho_\perp^{-4l}\E[\norm{u}^{4}]
\le C_{\rho_\perp}^{2}\rho_\perp^{-4l}(d-m+1)^{2}.
\]
The cross term is controlled by the triangle inequality for the $L^{2}$ norm of centred random variables, $\mathrm{Var}[X+Y]\le\bigl(\sqrt{\mathrm{Var}[X]}+\sqrt{\mathrm{Var}[Y]}\bigr)^{2}$,
whence
\[
\mathrm{Var}[S_1]
\le
\Bigl(\sqrt{2}\,\|A_{m,\perp}^{+}\|_F
      + C_{\rho_\perp}\rho_\perp^{-2l}(d-m+1)\Bigr)^{2}.
\]
Dividing by $N$ and using
$\|A_{m,\perp}^{+}\|_F^{2}\le(d-m)\|A_{m,\perp}^{+}\|_2^{2}\le(d-m)\lambda_d^{-2}$
(by \eqref{eq:interlace}, $\mu_{d-m}\ge\lambda_d$) establishes
\eqref{eq:variance_bound}.\hfill$\checkmark$
 
\medskip
\noindent\textbf{Step 4: Almost sure convergence.}
For fixed $l$, the $S_i$ are i.i.d.\ with finite mean and variance
(Steps 2--3), so the strong law of large numbers gives
$\hat{t}_{m,\perp}\xrightarrow{\mathrm{a.s.}}\E[S_1]$, and
$|\E[S_1]-t_{m,\perp}|\le C_{\rho_\perp}\rho_\perp^{-2l}(d-m)$ by Step~2. For the joint statement, let $l_N\to\infty$ be arbitrary and fix $\varepsilon>0$. By the concentration inequality \eqref{eq:hw_tail} established in the proof of Corollary~\ref{cor:pslq_rate} below (which does not rely on this step), there are constants such that, for all $N$ large enough that the bias $C_{\rho_\perp}\rho_\perp^{-2l_N}(d-m)\le \frac{\varepsilon}{2}$,
\[
\Pr\bigl(|\hat{t}_{m,\perp}^{(N,l_N)}-t_{m,\perp}|>\varepsilon\bigr)
\;\le\; 4\exp(-c_\varepsilon N)
\]
for a constant $c_\varepsilon>0$ depending on $\varepsilon$, $d-m$ and
$\lambda_d$ but not on $N$. The right-hand side is summable in $N$, so by the Borel-Cantelli lemma $|\hat{t}_{m,\perp}^{(N,l_N)}-t_{m,\perp}|>\varepsilon$ occurs only finitely often almost surely. Since $\varepsilon>0$ was arbitrary, \eqref{eq:as_convergence} follows.\hfill$\square$
\end{proof}

\begin{remark}
\label{rem:pslq_role}
Theorem~\ref{thm:pslq_consistency} concerns the compressed trace
$\mathrm{tr}(T_{m,\perp}^{-1})$, which is the input to the coupling correction of Section~\ref{sec:sperp_estimation}, not the bulk magnitude $\omega(m)$ of \eqref{eq:omega_def} directly. The corrected trace $\mathrm{tr}(S_{\perp}^{-1})$ is obtained from the P-SLQ estimate by the deterministic Sherman-Morrison correction \eqref{eq:trace_SM}, whose ingredients come from a single boundary probe; since that correction is an exponentially accurate function of the boundary quadratic forms, the consistency and rate guarantees above transfer to $\mathrm{tr}(S_{\perp}^{-1})$ and hence to $\omega(m)$.
\end{remark}

\begin{corollary}[Sample Complexity of P-SLQ]
\label{cor:pslq_rate}
Let $\hat{t}_{m,\perp}$ be the P-SLQ estimator of Algorithm~\ref{alg:pslq} and $t_{m,\perp}=\tr\bigl((\Pmp A\Pmp)^{+}\bigr)$. There is an absolute constant $C>0$ such that, for every target accuracy $0<\varepsilon\le(d-m)\lambda_d^{-1}$ and confidence level $\delta\in\,]0,1[$, the choices
\begin{align}
l &\ge \frac{1}{2\log\rho_{\perp}}
\log\left(\frac{4\,C_{\rho_{\perp}}\,(d-m)}{\varepsilon}\right)
=\mathcal{O}\bigl(\log(1/\varepsilon)\bigr),
\label{eq:l_sufficient}\\[4pt]
N &\ge C\,
\max\left\{\frac{(d-m)\lambda_{d}^{-2}}{\varepsilon^{2}},
\frac{\lambda_{d}^{-1}}{\varepsilon}\right\}
\log\left(\frac{4}{\delta}\right)
=\mathcal{O}\bigl(\varepsilon^{-2}\log(1/\delta)\bigr)
\label{eq:N_sufficient}
\end{align}
guarantee $\Pr\bigl(|\hat{t}_{m,\perp}-t_{m,\perp}|\le\varepsilon\bigr)\ge1-\delta$.
\end{corollary}
 
\begin{proof}
Write $S_i=\widetilde S_i+\norm{u_i}^{2}\varepsilon_l^{(i)}$ with
$\widetilde S_i=\xi_i^{\top}A_{m,\perp}^{+}\xi_i$ and
$|\varepsilon_l^{(i)}|\le C_{\rho_\perp}\rho_\perp^{-2l}$ by
\eqref{eq:quad_error}, and let $\bar{\widetilde S}=\frac1N\sum_i\widetilde S_i$ and $\bar U=\frac1N\sum_i\norm{u_i}^{2}$. Since $\E[\widetilde S_1]=\tr(A_{m,\perp}^{+})=t_{m,\perp}$ exactly by Hutchinson's identity \eqref{eq:hutch_applied}, the quadrature bias is confined to the second term, and
\begin{equation}
|\hat{t}_{m,\perp}-t_{m,\perp}|
\;\le\;
|\bar{\widetilde S}-t_{m,\perp}|
+ C_{\rho_\perp}\rho_\perp^{-2l}\,\bar U .
\label{eq:cor_split}
\end{equation}
 
\emph{First term: Hanson-Wright.} Stacking the probes,
$\bar{\widetilde S}=\frac1N\,\zeta^{\top}\bigl(I_N\otimes A_{m,\perp}^{+}\bigr)\zeta$
with $\zeta\sim\mathcal{N}(0,I_{Nd})$. The Hanson--Wright inequality
\citep{rudelson2013hanson} applied to the block-diagonal matrix
$I_N\otimes A_{m,\perp}^{+}$, which satisfies
$\|I_N\otimes A_{m,\perp}^{+}\|_F^{2}=N\|A_{m,\perp}^{+}\|_F^{2}$ and
$\|I_N\otimes A_{m,\perp}^{+}\|_2=\|A_{m,\perp}^{+}\|_2$, gives for an absolute
constant $c>0$
\begin{equation}
\Pr\bigl(|\bar{\widetilde S}-t_{m,\perp}|>\tfrac{\varepsilon}{2}\bigr)
\le
2\exp\left(-c\,N\min\left\{
\frac{\varepsilon^{2}}{\|A_{m,\perp}^{+}\|_F^{2}},
\frac{\varepsilon}{\|A_{m,\perp}^{+}\|_2}\right\}\right).
\label{eq:hw_tail}
\end{equation}
Using $\|A_{m,\perp}^{+}\|_F^{2}\le(d-m)\lambda_d^{-2}$ and
$\|A_{m,\perp}^{+}\|_2\le\lambda_d^{-1}$, the choice \eqref{eq:N_sufficient} with $C=2/c$ makes the right-hand side at most $\delta/2$.
 
\emph{Second term: control of $\bar U$.} $\bar U$ is itself an average of Gaussian quadratic forms with matrix $\Pmp$, $\E[\bar U]=d-m$,
$\|\Pmp\|_F^{2}=d-m$ and $\|\Pmp\|_2=1$, so \eqref{eq:hw_tail} with
$A_{m,\perp}^{+}$ replaced by $\Pmp$ and threshold $d-m$ gives
\[
\Pr\bigl(\bar U>2(d-m)\bigr)
\le 2\exp\bigl(-c\,N(d-m)\bigr)
\le\frac{\delta}{2},
\]
the last inequality because \eqref{eq:N_sufficient} together with
$\varepsilon\le(d-m)\lambda_d^{-1}$ implies
$N\ge C\log(4/\delta)\,(d-m)\lambda_d^{-2}/\varepsilon^{2}
\ge C\log(4/\delta)/(d-m)$. On the complementary event
$\{\bar U\le2(d-m)\}$, the choice \eqref{eq:l_sufficient} yields
$\rho_\perp^{-2l}\le\varepsilon/(4C_{\rho_\perp}(d-m))$ and hence
$C_{\rho_\perp}\rho_\perp^{-2l}\,\bar U\le\varepsilon/2$.
 
\emph{Union bound.} With probability at least $1-\delta$ both events hold, and \eqref{eq:cor_split} gives $|\hat{t}_{m,\perp}-t_{m,\perp}|\le\varepsilon$.
 
The restriction $\varepsilon\le(d-m)\lambda_d^{-1}$ is harmless: since
$0\le t_{m,\perp}\le(d-m)\lambda_d^{-1}$, any larger $\varepsilon$ is met by the trivial guarantee.
\end{proof}

\subsection{Estimating the Coupling-Corrected Magnitude}
\label{sec:sperp_estimation}

Theorem~\ref{thm:pslq_consistency} estimates $\tr(T_{m,\perp}^{-1})$, the trace of the \emph{compressed} inverse. The coupling-corrected covariance \eqref{eq:Sigma_coupled} instead requires the trace of
$S_{\perp}^{-1}=(T_{m,\perp}-\gamma\,e_1e_1^{\top})^{-1}$, together with the boundary scalars $(S_{\perp}^{-1})_{11}$ and $\sigma_u$ and the coupling direction $q_{\perp}$. All of these follow from the P-SLQ estimate by a single rank-one (Sherman-Morrison) correction, at the cost of one deterministic probe started at the boundary direction $e_1=q_{m+1}$.

\paragraph{Boundary quadratic forms.}
Run the reprojected Lanczos of Theorem~\ref{thm:complement_invariance} from the boundary direction $q_{m+1}$, producing an $l\times l$ tridiagonal $\Theta$. Because the run starts at $e_1$, the rank-one downdate acts on the first Krylov coordinate, so $S_{\perp}$ is represented in these coordinates by $\Theta_{\gamma}=\Theta-\gamma\,e_1e_1^{\top}$, and Gauss quadrature gives
\begin{equation}
a_1=e_1^{\top}\Theta^{-1}e_1\approx(T_{m,\perp}^{-1})_{11},
\quad
a_2=e_1^{\top}\Theta^{-2}e_1\approx(T_{m,\perp}^{-2})_{11},
\quad
p_k=e_1^{\top}\Theta_{\gamma}^{-k}e_1\approx(S_{\perp}^{-k})_{11},
\label{eq:boundary_forms}
\end{equation}
each accurate to $O(\rho_{\perp}^{-2l})$ by Theorem~\ref{thm:gauss_lanczos}.

\begin{proposition}[Sherman-Morrison correction]
\label{prop:sherman_morrison}
Let $\gamma=\beta_m^{2}(T_m^{-1})_{mm}$ with $\gamma\,(T_{m,\perp}^{-1})_{11}<1$.
Then
\begin{align}
\tr\!\bigl(S_{\perp}^{-1}\bigr)
&=\tr\!\bigl(T_{m,\perp}^{-1}\bigr)
+\frac{\gamma\,(T_{m,\perp}^{-2})_{11}}{1-\gamma\,(T_{m,\perp}^{-1})_{11}},
\label{eq:trace_SM}\\[2pt]
\bigl(S_{\perp}^{-1}\bigr)_{11}
&=\frac{(T_{m,\perp}^{-1})_{11}}{1-\gamma\,(T_{m,\perp}^{-1})_{11}},
\label{eq:s11_SM}\\[2pt]
Q_{m,\perp}S_{\perp}^{-1}e_1
&=\frac{Q_{m,\perp}T_{m,\perp}^{-1}e_1}{1-\gamma\,(T_{m,\perp}^{-1})_{11}} .
\label{eq:g_SM}
\end{align}
\end{proposition}

\begin{proof}
Since $S_{\perp}=T_{m,\perp}-\gamma\,e_1e_1^{\top}$ is a rank-one downdate, the Sherman-Morrison identity gives
\[
S_{\perp}^{-1}
=T_{m,\perp}^{-1}
+\frac{\gamma\,T_{m,\perp}^{-1}e_1e_1^{\top}T_{m,\perp}^{-1}}
{1-\gamma\,e_1^{\top}T_{m,\perp}^{-1}e_1}.
\]
Taking the trace and using
$\tr(T_{m,\perp}^{-1}e_1e_1^{\top}T_{m,\perp}^{-1})
=e_1^{\top}T_{m,\perp}^{-2}e_1=(T_{m,\perp}^{-2})_{11}$ yields
\eqref{eq:trace_SM}. Pairing the identity with $e_1$ on both sides gives
$(S_{\perp}^{-1})_{11}
=(T_{m,\perp}^{-1})_{11}
+\frac{\gamma\,(T_{m,\perp}^{-1})_{11}^{2}}{1-\gamma(T_{m,\perp}^{-1})_{11}}$, which simplifies to \eqref{eq:s11_SM}. Applying the identity to $e_1$ and mapping to full space by $Q_{m,\perp}$ gives
$Q_{m,\perp}S_{\perp}^{-1}e_1
=Q_{m,\perp}T_{m,\perp}^{-1}e_1\bigl(1+\frac{\gamma(T_{m,\perp}^{-1})_{11}}{
1-\gamma(T_{m,\perp}^{-1})_{11}}\bigr)$, which is \eqref{eq:g_SM}.
\end{proof}

\paragraph{Assembling the estimator.}
Combining P-SLQ with \eqref{eq:boundary_forms}--\eqref{eq:g_SM}:
\begin{equation}
\widehat{\tr}\bigl(S_{\perp}^{-1}\bigr)
=\widehat{t}_{\perp}+\frac{\gamma\,a_2}{1-\gamma\,a_1},
\qquad
s_{11}=\frac{a_1}{1-\gamma\,a_1},
\qquad
\sigma_u=\frac{p_3}{p_2},
\qquad
\|q_{\perp}\|_2^{2}=p_2,
\label{eq:assembled}
\end{equation}
where $\widehat{t}_{\perp}$ is the P-SLQ estimate of $\tr(T_{m,\perp}^{-1})$. The complement magnitude is then $\omega(m)=\frac{\widehat{\tr}(S_{\perp}^{-1})-\sigma_u}{d-m-1}$
\eqref{eq:omega_rest}. The coupling direction is obtained from a single
conjugate-gradient solve of $(P_m^{\perp}AP_m^{\perp})x=\beta_mq_{m+1}$, giving $Q_{m,\perp}T_{m,\perp}^{-1}e_1=\frac{x}{\beta_m}$, followed by the scalar rescaling \eqref{eq:g_SM},
\begin{equation}
q_{\perp}=\frac{\frac{x}{\beta_m}}{1-\gamma\,a_1},
\qquad u=\frac{q_{\perp}}{\|q_{\perp}\|_2} .
\label{eq:g_from_cg}
\end{equation}
The additional cost over P-SLQ is one boundary probe ($l$ products) and one CG solve ($l_{\mathrm{cg}}$ products); the P-SLQ consistency guarantees of Theorem~\ref{thm:pslq_consistency} carry over to $\widehat{\tr}(S_{\perp}^{-1})$ since the correction is a deterministic, exponentially accurate function of the boundary forms \eqref{eq:boundary_forms}.


The following Algorithm~\ref{alg:pslq} collects the construction of Section~\ref{sec:cu} and the estimators of Section~\ref{sec:pslq} into a single procedure.

\begin{algorithm}[H]
\caption{P-SLQ with exact boundary coupling}
\label{alg:pslq}
\begin{algorithmic}[1]
\Require matrix $A$ (implicit, via matrix-vector products), vector $v$,
         Krylov dimension $m$, quadrature depth $l$, probes $N$,
         CG iterations $l_{\mathrm{cg}}$, sample count $n_s$, centre
         $\theta_{\star}$
\vspace{0.4em}
\State \textbf{Phase 1: resolved basis and boundary data}
\State $q_1 \gets \frac{v}{\norm{v}}$; run $m$-step Lanczos on $A$ from $q_1$
       $\;\to\;$ $\Qm$, $\Tm$, residual norm $\beta_m$, next vector $q_{m+1}$
\State $\gamma \gets \beta_m^{2}\,(\Tm^{-1})_{mm}$
       \Comment{coupling scalar, eq.~\eqref{eq:Sperp}}
\State $q_{\parallel} \gets \beta_m\,\Qm \Tm^{-1} e_m$;\quad $\hat q \gets \frac{q_{\parallel}}{\norm{q_{\parallel}}}$
       \Comment{resolved coupling vector, eq.~\eqref{eq:coupling_pair}}
\vspace{0.4em}
\State \textbf{Phase 2a: bulk trace by P-SLQ}
\For{$i = 1, \dots, N$}
  \State $\xi_i \sim \N(0, I_d)$;\quad
         $u_i \gets \xi_i - \Qm(\Qm^\top \xi_i)$
  \State run $l$-step reprojected Lanczos on $A$ from $ \frac{u_i}{\norm{u_i}}$
         $\;\to\;\Theta_l^{(i)}$
  \State $\nu_i \gets e_1^\top \bigl(\Theta_l^{(i)}\bigr)^{-1} e_1$
\EndFor
\State $\widehat{t}_{\perp} \gets \frac{1}{N}\sum_{i=1}^N
       \norm{u_i}^2\, \nu_i$
       \Comment{$\approx\tr(T_{m,\perp}^{-1})$,
       Theorem~\ref{thm:pslq_consistency}}
\vspace{0.4em}
\State \textbf{Phase 2b: boundary probe}
\State run $l$-step reprojected Lanczos on $A$ from $q_{m+1}$
       $\;\to\;\Theta$
\State $\Theta_{\gamma} \gets \Theta - \gamma\,e_1 e_1^\top$
       \Comment{$S_{\perp}$ in the probe's coordinates}
\State $a_1 \gets e_1^\top \Theta^{-1} e_1$;\quad
       $a_2 \gets e_1^\top \Theta^{-2} e_1$
       \Comment{eq.~\eqref{eq:boundary_forms}}
\State $p_2 \gets e_1^\top \Theta_{\gamma}^{-2} e_1$;\quad
       $p_3 \gets e_1^\top \Theta_{\gamma}^{-3} e_1$;\quad
       $p_4 \gets e_1^\top \Theta_{\gamma}^{-4} e_1$
    
\vspace{0.4em}
\State \textbf{Phase 2c: coupling-corrected magnitudes}
\State $\widehat{\tr}\bigl(S_{\perp}^{-1}\bigr) \gets
       \widehat{t}_{\perp} + \dfrac{\gamma\, a_2}{1-\gamma\, a_1}$
       \Comment{Prop.~\ref{prop:sherman_morrison}}
\State $s_{11} \gets \dfrac{a_1}{1-\gamma\,a_1}$;\quad
       $\sigma_u \gets \frac{p_3}{p_2}$;\quad
       $\norm{q_{\perp}}^2 \gets p_2$
       \Comment{eq.~\eqref{eq:assembled}}
\State $\omega(m) \gets
       \frac{\widehat{\tr}(S_{\perp}^{-1})-\sigma_u}{d-m-1}$
       \Comment{eq.~\eqref{eq:omega_rest}}
\State $\tau^{2} \gets \dfrac{p_4}{p_2}-\Bigl(\dfrac{p_3}{p_2}\Bigr)^{2}$
       \Comment{cross-term $\|c\|_2^{2}$,
       Prop.~\ref{prop:total_quality}(iii); a posteriori error diagnostic}
       
\State $C_2 \gets \begin{pmatrix}
       s_{11}\norm{q_{\parallel}}^2 & -\norm{q_{\parallel}}\norm{q_{\perp}}\\[1pt]
       -\norm{q_{\parallel}}\norm{q_{\perp}} & \sigma_u\end{pmatrix}$;\quad
       compute $C_2^{\frac{1}{2}}$
       \Comment{eq.~\eqref{eq:C2}; PSD by Lemma~\ref{lem:C2_psd}}
\vspace{0.4em}
\State \textbf{Phase 2d: coupling direction}
\State solve $\bigl(P_m^{\perp}AP_m^{\perp}\bigr)\,x = \beta_m\, q_{m+1}$
       by $l_{\mathrm{cg}}$ CG iterations, operator
       $y \mapsto P_m^{\perp}\!\bigl(A\,(P_m^{\perp}y)\bigr)$
\State $q_{\perp} \gets \dfrac{\frac{x}{\beta_m}}{1-\gamma\,a_1}$;\quad
       $u \gets \frac{q_{\perp}}{\norm{q_{\perp}}}$
       \Comment{eq.~\eqref{eq:g_from_cg}}
\vspace{0.4em}
\State \textbf{Phase 3: sampling (Proposition~\ref{prop:sampler})}
\For{$s = 1, \dots, n_s$}
  \State $z_1 \sim \N(0,I_m)$;\quad $z_c \sim \N(0,I_2)$;\quad
         $z_2 \sim \N(0,I_d)$
  \State $\theta^{(s)} \gets \theta_{\star}
         + \Qm \Tm^{-\frac{1}{2}} z_1
         + [\,\hat q\ \ u\,]\,C_2^{\frac{1}{2}} z_c$
         $+\ \sqrt{\omega(m)}\,\bigl(z_2 - \Qm(\Qm^{\top}z_2) - u\,(u^{\top}z_2)\bigr)$
\EndFor
\Statex \textit{Total matrix-vector products:}
        $m + Nl + l + l_{\mathrm{cg}}$.
\end{algorithmic}
\end{algorithm}

\section{Experiments}
\label{sec:experiments}

We validate the method in two steps: we have first the section~\ref{sec:exp_pslq} tests the P-SLQ estimator in three terms: its consistency, the role of re-projection, and its accuracy against the unbiased solver ULISSE and some state-of-the-art technics for trace estimation at matched matrix-vector budgets; and then the section~\ref{sec:exp_laplace} applies the full pipeline to the Laplace approximation in Bayesian deep learning.

\subsection{P-SLQ Estimator: Consistency, Reprojection, and Trace Accuracy}
\label{sec:exp_pslq}

\paragraph{Consistency.}
We first verify the two guarantees of Theorem~\ref{thm:pslq_consistency}, for that we use three spectra cases to represent different setting of complement conditioning that governs the isotropic approximation (Theorem~\ref{thm:iso}):
\begin{itemize}
\item \emph{Bulk + outliers}: Where a few large outliers are present over a large range of values. This most often represents the Hessian matrix of deep networks, whose spectrum consists of a small number of outliers, on the order of the number of classes, over a large set. \citep{sagun2017empirical, papyan2020traces, ghorbani2019investigation}. Once Lanczos resolves the outliers, the remaining complement is  well-conditioned ($\kappa_S^{\perp}(m)$ small), the case  where   isotropic approximation  is accurate.

\item \emph{Power-law} ($\lambda_i\propto i^{-p}$): With   a smoothly spread spectrum. This is the case of kernel and integral operators and of discretised elliptic PDEs, and it is the setting in which the complement keeps substantial spectral mass at every $m$.

\item \emph{Exponential} ($\lambda_i\propto e^{-ci}$): With a rapid geometric decay, so the inverse concentrates in a few directions and the complement remains ill-conditioned even at large $m$ ($\kappa_S^{\perp}(m)$ large). This is the complicated  case for isotropic approximation.
\end{itemize}
These cases correspond to an increasing allocation of the complementary inverse spectrum, and therefore to an increasing allocation  error, allowing us to probe the dependence $\kappa_S^{\perp}(m)$ predicted by our theory.

Fixing $A$ and the Krylov dimension $m$, we measure the error of $\widehat{t}_{m,\perp}$ against the true trace $\mathrm{tr}(T_{m,\perp}^{-1})$ as a function of the quadrature depth $l$ and the number of probes $N$ (Figure~\ref{fig:exp_consistency}). The bias decays geometrically in $l$, at the rate $\rho_{\perp}^{-2l}$ of Theorem~\ref{thm:gauss_lanczos}, and the relative variance decays as $N^{-1}$, matching the Monte Carlo rate of \eqref{eq:variance_bound}. Both work  for  the three spectra, with the fastest bias decay on the well-conditioned bulk+outliers regime and the slowest on the exponential, as expected from the complement conditioning.

\begin{figure}[!htbp]
\centering
\includegraphics{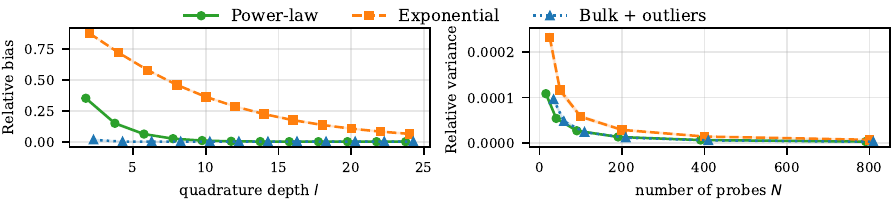}
\caption{P-SLQ consistency. \emph{Left}: relative bias
$\frac{|\mathbb{E}[\hat t_{m,\perp}]-t_{m,\perp}|}{t_{m,\perp}}$ versus quadrature depth $l$ (geometric decay). \emph{Right}: relative variance versus number of probes $N$ ($\propto N^{-1}$). Curves are means over independent matrices with $\pm$ one standard deviation bands.}
\label{fig:exp_consistency}
\end{figure}
\FloatBarrier
\paragraph{The role of reprojection.} P-SLQ reprojects every Lanczos step onto $\mathrm{range}(P_{m}^{\perp})$. This is what the consistency proof requires the iteration to stay in the complement, and it is not implied by reorthogonalisation, which only enforces orthogonality within the quadrature basis. Figure~\ref{fig:exp_reproj} isolates its effect: both variants use identical full reorthogonalisation and differ only in the per-step reprojection. Without it the iterates can turn back to  the resolved directions, which the operator then re-amplifies, and the bias grows relative to P-SLQ as the quadrature depth increases.
We report this on the power-law spectrum, and the reason is worth giving. An iterate can only drift back if the complement holds large and spread eigenvalues, since it is the operator on the complement that re-amplifies a leaked component. On the exponential and bulk+outliers spectra the complement is either concentrated or well separated from the resolved block, the iterates barely drift, and the two variants give the same result up to noise. The  re-projection  matters in the regime where the complement matters, and costs nothing in the others.

\begin{figure}[!htbp]
\centering
\includegraphics[width=\linewidth,height=0.3\textheight,keepaspectratio]{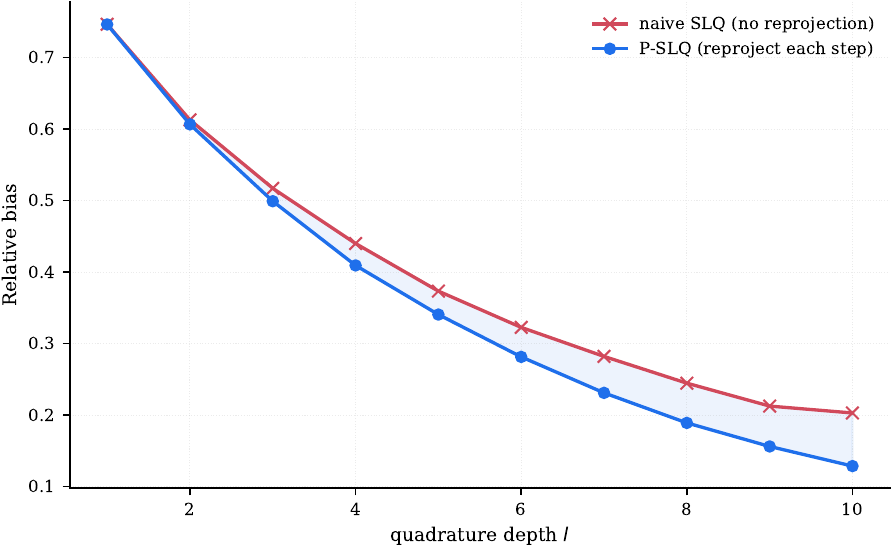}
\caption{Effect of reprojection on the complement-trace estimate (power-law
spectrum, $d=2000$, $m=20$). Both variants use full reorthogonalisation; they
differ only in the per-step reprojection onto $\mathrm{range}(P_{m}^{\perp})$.
Without reprojection the iterates drift toward the resolved directions and the relative bias grows with the quadrature depth $l$.}
\label{fig:exp_reproj}
\end{figure}
\FloatBarrier
\paragraph{Trace estimation.}
We compare our method  P-SLQ against standard Hutchinson+CG, classic SLQ, and the state-of-the-art stochastic trace estimators Hutch++,  XTrace \citep{persson2022xtrace} and the kernel-5d from  spatial Gaussian process covariance matrix generated from $n=1000$ points uniformly sampled in $\mathbb{R}^5$, using an exponential kernel $A_{ij} = \exp(-\|x_i - x_j\|_2 / \ell) + \sigma_{\mathrm{GP}}^2 \delta_{ij}$ with lengthscale $\ell = 0.3\sqrt{5}$ and noise variance $\sigma_{\mathrm{GP}}^2 = 0.01$ ( $\delta_{ij}$ is the Kronecker delta), alongside the unbiased Russian-roulette solver ULISSE \citep{filippone2015enabling} on the estimation of $\mathrm{tr}(A^{-1})$ at the same matrix-vector budgets (Figure~\ref{fig:exp_ulisse}).  What we can see is  P-SLQ achieves the lowest relative RMSE across all spectral regimes by trading a small, exponentially decaying bias for an $O(N^{-1})$ variance independent of $m$, reusing a single Krylov basis across all probes and completely bypassing multiple costly linear solves.

\begin{figure}[!htbp]
\centering
\includegraphics{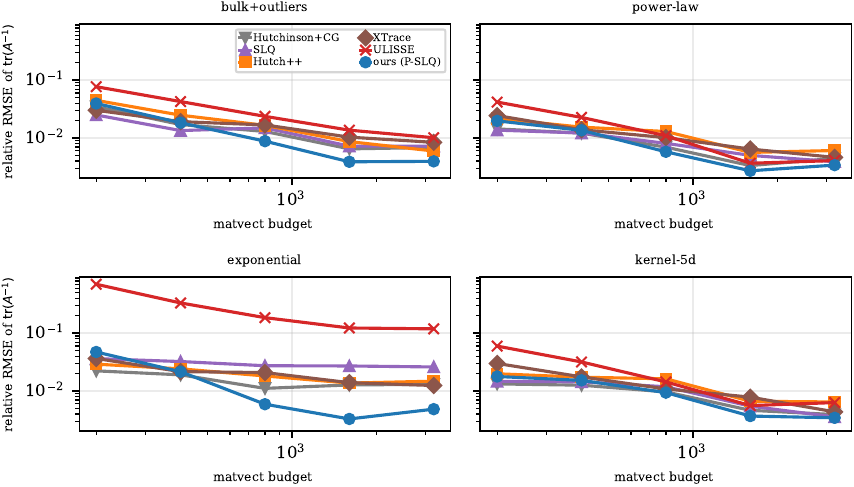}
\caption{Relative RMSE of $\mathrm{tr}(A^{-1})$ versus matrix-vector budget for P-SLQ,  ULISSE \citep{filippone2015enabling} and some state-of-the-art technics, on the three spectra ($d=1000$, $\kappa \approx 10^{5}$) using synthetic matrices, plus the kernel-5d}
\label{fig:exp_ulisse}
\end{figure}
\FloatBarrier
\subsection{Application for Laplace Approximation in Bayesian deep learning}
\label{sec:exp_laplace}

We apply the corrected covariance to the Laplace approximation, where $A$ is the damped generalized Gauss-Newton matrix and $ A^{-1}$ is the posterior covariance. The main goal is to bring a matrix-free posterior distribution closer to the exact Laplace posterior distribution using the krylov complement restoration, compare to existing  structural approximations. We measure this directly by calculating the exact predictive covariance and comparing it to the corrected covariance.

\subsubsection{Setup}
\label{sec:setup}
We use two image classification benchmarks: a convolutional network on FashionMNIST ($p = 11,878$), where the full Laplace posterior can be formed and inverted, and a  style of ResNet-9 network without BatchNorm on CIFAR-10 ($p = 672,890$), where it cannot. BatchNorm is removed so that KFAC and the marginal-likelihood fit are available in \texttt{laplace-torch}. The CIFAR-10 network is almost the same in terms of size to the ResNet-9 of \citet{faller2025low} but is not their architecture, so our CIFAR-10 results are not a reproduction of their Figure 2. All results are reported as mean +/- standard error over 5 seeds, across seeds only the stochastic components vary (Lanczos start vector, P-SLQ probes, predictive draws).

The GGN is assembled from a subset of $200$ training points and rescaled by $50000/200$ to match the full training set. The prior precision $\lambda$ is fitted by marginal likelihood on the unscaled $200$-point operator and multiplied by the same factor, it is  not the marginal-likelihood optimum of the rescaled operator, and we use it because refitting against the rescaled GGN would cost a full-data marginal likelihood. We obtain $\lambda$ approximate $350$ on FashionMNIST and $\lambda$ approximate $2900$ on CIFAR-10. The KFAC surrogate used to construct the low-rank-KFAC subspace is fitted on the same $200$ points without rescaling, so its curvature is small relative to the prior compared with the operator that all methods are then evaluated against. This works against low-rank KFAC, so our comparison is conservative rather than favourable

\subsubsection{Metrics}
\label{sec:metrics}

\paragraph{Fidelity.}
We report the per-point Gaussian KL divergence of \citet[eq.~19]{faller2025low}. Write $\Sigma_X^{(i)}$ and $\widehat{\Sigma}_X^{(i)}$ for the $C \times C$ diagonal blocks of the exact and the approximate predictive covariance at test point $i$, and average over the $n$ test points:
\begin{equation}
  \overline{\mathrm{KL}}
  \;=\; \frac{1}{n}\sum_{i=1}^{n}
  \mathrm{KL}\!\Big(
    \mathcal{N}\big(0,\Sigma_X^{(i)}\big)
    \,\Big\|\,
    \mathcal{N}\big(0,\widehat{\Sigma}_X^{(i)}\big)
  \Big),
  \label{eq:kl-perpoint}
\end{equation}
where, for symmetric positive definite $\Sigma_0,\Sigma_1 \in \mathbb{R}^{C\times C}$,
\begin{equation}
  \mathrm{KL}\big(\mathcal{N}(0,\Sigma_0) \,\|\, \mathcal{N}(0,\Sigma_1)\big)
  = \tfrac{1}{2}\Big(
      \operatorname{tr}\big(\Sigma_1^{-1}\Sigma_0\big) - C
      + \log\det\Sigma_1 - \log\det\Sigma_0
    \Big).
  \label{eq:kl-gauss}
\end{equation}
All  is computed matrix-free. The KL requires $\Sigma_X = J_X A^{-1} J_X^{\top}$, an $nC \times nC$ matrix, and not $A^{-1}$ itself,
so solving $A X = J_X^{\top}$ by conjugate gradients and forming $J_X X$ gives it exactly. On CIFAR-10 we get the worst relative residual below $10^{-6}$. This is the direction used by \citet{faller2025low}, and it penalises an approximation for assigning too little variance. Since every rank-s  projection does not assign a variance to its subspace, we also use the log-trace, which is not directional. Blocks with a singular approximate covariance give an infinite KL, we exclude them and report the excluded fraction.

\paragraph{Predictive performance.}
We report the test negative log-likelihood, the expected calibration error with $15$ bins, and the Brier score, all under the probit predictive.

\subsubsection{Subspace baselines}
\label{sec:subspace_baselines}

To evaluate our corrected covariance, we compare it  against the data‑driven subspace projections of \citet{faller2025low}.  Their framework constructs a rank‑$s$ projection of the predictive covariance, $\Sigma_{P,X} = J_X P (P^\top A P)^{-1} P^\top J_X^\top$, where the projection matrix $P\in\mathbb{R}^{p\times s}$ is built from a low‑rank approximation of the GGN inverse.  Specifically, let $\Psi_{\mathrm{approx}}$ be a scalable surrogate for $A^{-1}$ (e.g.~the diagonal of $A^{-1}$ or the KFAC approximation).  Given a subset of training inputs $X'$, the top‑$s$ eigenvectors $U_s$ of the matrix $J_{X'} \Psi_{\mathrm{approx}} J_{X'}^\top$ are computed, and the projection is set to $P = \Psi_{\mathrm{approx}} J_{X'}^\top U_s$.  The resulting predictive
covariance has rank at most $s$, assigning exactly zero variance to the
remaining $p-s$ directions. While the per‑point KL divergence \citep[eq.~19]{faller2025low} remains finite, the projection methods  undercover the predictive uncertainty, as reflected by their KL values
(figure~\ref{fig:fashionmnist_cifar10_scaling} and Table~\ref{tab:fullrank_kl_logtr}).

We evaluate four variants of this construction: \texttt{lowrank‑Diagonal} (using the diagonal of $A^{-1}$), \texttt{lowrank‑KFAC} (using the KFAC approximation), and the theoretically optimal projection \texttt{lowrankopt‑GGN} (available only on FashionMNIST, where the full GGN can be formed and decomposed).  To isolate the effect of our isotropic complement correction, we also include a low-rank-plus-shift baseline, which augments a subspace with a $1/\lambda$ variance on its complement. On FashionMNIST it is built on the Theorem~1 optimum, and we write it \texttt{lowrankopt + $1/\lambda$ shift}. On CIFAR-10 that optimum is not available, so it is built on the best feasible subspace, \texttt{lowrank-KFAC}, and we write it \texttt{lowrank-KFAC + $1/\lambda$ shift}. The two rows are therefore not the same construction, and we label them separately. We also include the subset-based selection heuristic \texttt{subset-Magnitude}, which selects the $s$ parameters of largest absolute MAP value \citep{daxberger2021subnetwork}.

The Theorem~1 optimum $P^{\star} = \Psi J_{X}^{\top} U_s$ is reported on FashionMNIST, where the full GGN can be formed and inverted. On CIFAR-10 it is omitted: with $\Psi$ available only through matrix-vector products, it requires one conjugate-gradient solve per row of $J_{X}$, which was beyond our compute budget. It is a reference ceiling rather than a competing method, but the omission means our CIFAR-10 comparison is against the feasible methods of
\citet{faller2025low}, not against their optimum, and we do not claim
otherwise.

\subsubsection{Predictive performance}

The tables~\ref{tab:fashionmnist_results} and~\ref{tab:cifar10_results} show the  the test negative log‑likelihood (NLL), the expected calibration error (ECE) and the brier score  for a range of subspace dimensions $s$.  Full‑rank Laplace baselines (diagonal, KFAC, ELLA, and the exact full Laplace) are shown at $s = p$.  The exact full Laplace is used as the reference, its NLL beats the MAP on FashionMNIST, confirming consistency of the GGN scale and marginal‑likelihood fit.

\begin{table}[!htbp]
\centering
\caption{FashionMNIST CNN ($p=11{,}878$): predictive performance vs.\ subspace dimension $s$.  Mean $\pm$ s.e.m.\ over 5 seeds.}
\label{tab:fashionmnist_results}
\resizebox{\textwidth}{!}{%
\begin{tabular}{lcccc}
\toprule
\textbf{Method} & $s$ & \textbf{NLL} $\downarrow$ & \textbf{ECE} $\downarrow$ & \textbf{Brier} $\downarrow$ \\
\midrule
\multicolumn{5}{c}{\textit{Subspace methods}} \\
\midrule
subset‑Magnitude       & 120 & $0.2609 \pm 0.0046$ & $0.0273 \pm 0.0022$ & $0.1331 \pm 0.0015$ \\
lowrank‑Diagonal       & 120 & $0.2600 \pm 0.0046$ & $0.0274 \pm 0.0025$ & $0.1330 \pm 0.0014$ \\
lowrank‑KFAC           & 120 & $0.2562 \pm 0.0043$ & $0.0226 \pm 0.0025$ & $0.1323 \pm 0.0014$ \\
lowrankopt‑GGN         & 120 & $0.2541 \pm 0.0041$ & $0.0189 \pm 0.0024$ & $0.1319 \pm 0.0014$ \\
lowrankopt + $1/\lambda$ shift & 120 & $0.3519 \pm 0.0101$ & $0.1155 \pm 0.0082$ & $0.1655 \pm 0.0048$ \\
\textbf{$\Sigma_m$ (ours)} & \textbf{120} & $\mathbf{0.2543 \pm 0.0023}$ & $\mathbf{0.0224 \pm 0.0021}$ & $\mathbf{0.1323 \pm 0.0011}$ \\
\midrule
subset‑Magnitude       & 230 & $0.2606 \pm 0.0046$ & $0.0277 \pm 0.0025$ & $0.1331 \pm 0.0015$ \\
lowrank‑Diagonal       & 230 & $0.2591 \pm 0.0046$ & $0.0265 \pm 0.0027$ & $0.1329 \pm 0.0014$ \\
lowrank‑KFAC           & 230 & $0.2552 \pm 0.0042$ & $0.0204 \pm 0.0026$ & $0.1321 \pm 0.0014$ \\
lowrankopt‑GGN         & 230 & $0.2535 \pm 0.0040$ & $0.0192 \pm 0.0024$ & $0.1319 \pm 0.0014$ \\
lowrankopt + $1/\lambda$ shift & 230 & $0.3516 \pm 0.0101$ & $0.1151 \pm 0.0082$ & $0.1654 \pm 0.0048$ \\
\textbf{$\Sigma_m$ (ours)} & \textbf{230} & $\mathbf{0.2520 \pm 0.0027}$ & $\mathbf{0.0187 \pm 0.0029}$ & $\mathbf{0.1317 \pm 0.0012}$ \\
\midrule
subset‑Magnitude       & 450 & $0.2599 \pm 0.0045$ & $0.0269 \pm 0.0025$ & $0.1330 \pm 0.0014$ \\
lowrank‑Diagonal       & 450 & $0.2578 \pm 0.0044$ & $0.0253 \pm 0.0027$ & $0.1327 \pm 0.0014$ \\
lowrank‑KFAC           & 450 & $0.2541 \pm 0.0041$ & $0.0189 \pm 0.0023$ & $0.1319 \pm 0.0014$ \\
lowrankopt‑GGN         & 450 & $0.2528 \pm 0.0040$ & $0.0180 \pm 0.0025$ & $0.1317 \pm 0.0014$ \\
lowrankopt + $1/\lambda$ shift & 450 & $0.3513 \pm 0.0101$ & $0.1147 \pm 0.0081$ & $0.1653 \pm 0.0048$ \\
\textbf{$\Sigma_m$ (ours)} & \textbf{450} & $\mathbf{0.2507 \pm 0.0030}$ & $\mathbf{0.0165 \pm 0.0016}$ & $\mathbf{0.1314 \pm 0.0013}$ \\
\midrule
subset‑Magnitude       & 670 & $0.2593 \pm 0.0045$ & $0.0264 \pm 0.0027$ & $0.1329 \pm 0.0014$ \\
lowrank‑Diagonal       & 670 & $0.2569 \pm 0.0044$ & $0.0247 \pm 0.0024$ & $0.1326 \pm 0.0014$ \\
lowrank‑KFAC           & 670 & $0.2534 \pm 0.0040$ & $0.0187 \pm 0.0024$ & $0.1318 \pm 0.0014$ \\
lowrankopt‑GGN         & 670 & $0.2523 \pm 0.0039$ & $0.0170 \pm 0.0025$ & $0.1316 \pm 0.0014$ \\
lowrankopt + $1/\lambda$ shift & 670 & $0.3511 \pm 0.0101$ & $0.1145 \pm 0.0081$ & $0.1653 \pm 0.0048$ \\
\textbf{$\Sigma_m$ (ours)} & \textbf{670} & $\mathbf{0.2502 \pm 0.0031}$ & $\mathbf{0.0151 \pm 0.0010}$ & $\mathbf{0.1313 \pm 0.0013}$ \\
\midrule
subset‑Magnitude       & 1000 & $0.2584 \pm 0.0045$ & $0.0254 \pm 0.0027$ & $0.1328 \pm 0.0014$ \\
lowrank‑Diagonal       & 1000 & $0.2559 \pm 0.0043$ & $0.0237 \pm 0.0026$ & $0.1324 \pm 0.0014$ \\
lowrank‑KFAC           & 1000 & $0.2528 \pm 0.0039$ & $0.0190 \pm 0.0024$ & $0.1317 \pm 0.0014$ \\
lowrankopt‑GGN         & 1000 & $0.2518 \pm 0.0039$ & $0.0163 \pm 0.0027$ & $0.1315 \pm 0.0014$ \\
lowrankopt + $1/\lambda$ shift & 1000 & $0.3509 \pm 0.0101$ & $0.1142 \pm 0.0081$ & $0.1652 \pm 0.0048$ \\
\textbf{$\Sigma_m$ (ours)} & \textbf{1000} & $\mathbf{0.2499 \pm 0.0032}$ & $\mathbf{0.0168 \pm 0.0005}$ & $\mathbf{0.1313 \pm 0.0013}$ \\
\midrule
\multicolumn{5}{c}{\textit{Full‑rank baselines ($s = p = 11{,}878$) and exact reference}} \\
\midrule
Diagonal LA            & $p$ & $0.2536 \pm 0.0032$ & $0.0152 \pm 0.0008$ & $0.1331 \pm 0.0010$ \\
KFAC LA                & $p$ & $0.2604 \pm 0.0020$ & $0.0260 \pm 0.0026$ & $0.1344 \pm 0.0007$ \\
ELLA (Nyström, $K=500$) & $p$ & $0.2543 \pm 0.0042$ & $0.0191 \pm 0.0022$ & $0.1320 \pm 0.0015$ \\
SWAG & $p$ & $0.2655 \pm 0.0054 $ & $0.0184 \pm 0.0024$ & $0.1391 \pm 0.0033$ \\
Exact Full LA          & --  & $0.2498 \pm 0.0033$ & $0.0154 \pm 0.0012$ & $0.1312 \pm 0.0013$ \\
MAP                    & --  & $0.2613 \pm 0.0047$ & $0.0275 \pm 0.0022$ & $0.1332 \pm 0.0015$ \\
\bottomrule
\end{tabular}%
}
\end{table}
\FloatBarrier
\begin{table}[H]
\centering
\caption{CIFAR‑10 ResNet‑9 ($p\approx672{,}890$): predictive performance vs.\ subspace dimension $s$.  Mean $\pm$ s.e.m.\ over 5 seeds.}
\label{tab:cifar10_results}
\resizebox{\textwidth}{!}{%
\begin{tabular}{lcccc}
\toprule
\textbf{Method} & $s$ & \textbf{NLL} $\downarrow$ & \textbf{ECE} $\downarrow$ & \textbf{Brier} $\downarrow$ \\
\midrule
\multicolumn{5}{c}{\textit{Subspace methods}} \\
\midrule
subset‑Magnitude       & 20  & $0.3758 \pm 0.0185$ & $0.0446 \pm 0.0013$ & $0.1796 \pm 0.0065$ \\
lowrank‑Diagonal       & 20  & $0.3755 \pm 0.0183$ & $0.0450 \pm 0.0015$ & $0.1796 \pm 0.0065$ \\
lowrank‑KFAC           & 20  & $0.3755 \pm 0.0184$ & $0.0450 \pm 0.0015$ & $0.1796 \pm 0.0065$ \\
lowrank-KFAC + $1/\lambda$ shift & 20 & $0.3614 \pm 0.0152$ & $0.0371 \pm 0.0043$ & $0.1811 \pm 0.0076$ \\
\textbf{$\Sigma_m$ (ours)} & \textbf{20}  & $\mathbf{0.3639 \pm 0.0148}$ & $\mathbf{0.0363 \pm 0.0027}$ & $\mathbf{0.1818 \pm 0.0073}$ \\
\midrule
subset‑Magnitude       & 50  & $0.3758 \pm 0.0185$ & $0.0452 \pm 0.0015$ & $0.1796 \pm 0.0065$ \\
lowrank‑Diagonal       & 50  & $0.3747 \pm 0.0182$ & $0.0452 \pm 0.0014$ & $0.1795 \pm 0.0065$ \\
lowrank‑KFAC           & 50  & $0.3741 \pm 0.0181$ & $0.0460 \pm 0.0023$ & $0.1794 \pm 0.0064$ \\
lowrank-KFAC + $1/\lambda$ shift & 50 & $0.3603 \pm 0.0152$ & $0.0343 \pm 0.0033$ & $0.1807 \pm 0.0076$ \\
\textbf{$\Sigma_m$ (ours)} & \textbf{50}  & $\mathbf{0.3622 \pm 0.0148}$ & $\mathbf{0.0349 \pm 0.0038}$ & $\mathbf{0.1812 \pm 0.0074}$ \\
\midrule
subset‑Magnitude       & 100 & $0.3756 \pm 0.0184$ & $0.0457 \pm 0.0014$ & $0.1796 \pm 0.0065$ \\
lowrank‑Diagonal       & 100 & $0.3734 \pm 0.0180$ & $0.0449 \pm 0.0017$ & $0.1794 \pm 0.0065$ \\
lowrank‑KFAC           & 100 & $0.3732 \pm 0.0180$ & $0.0445 \pm 0.0021$ & $0.1794 \pm 0.0065$ \\
lowrank-KFAC + $1/\lambda$ shift & 100 & $0.3599 \pm 0.0152$ & $0.0329 \pm 0.0031$ & $0.1805 \pm 0.0076$ \\
\textbf{$\Sigma_m$ (ours)} & \textbf{100} & $\mathbf{0.3601 \pm 0.0149}$ & $\mathbf{0.0337 \pm 0.0022}$ & $\mathbf{0.1803 \pm 0.0074}$ \\
\midrule
subset‑Magnitude       & 200 & $0.3754 \pm 0.0184$ & $0.0453 \pm 0.0010$ & $0.1796 \pm 0.0065$ \\
lowrank‑Diagonal       & 200 & $0.3720 \pm 0.0178$ & $0.0453 \pm 0.0038$ & $0.1793 \pm 0.0065$ \\
lowrank‑KFAC           & 200 & $0.3718 \pm 0.0178$ & $0.0454 \pm 0.0037$ & $0.1792 \pm 0.0065$ \\
lowrank-KFAC + $1/\lambda$ shift & 200 & $0.3595 \pm 0.0153$ & $0.0333 \pm 0.0031$ & $0.1803 \pm 0.0076$ \\
\textbf{$\Sigma_m$ (ours)} & \textbf{200} & $\mathbf{0.3584 \pm 0.0155}$ & $\mathbf{0.0328 \pm 0.0031}$ & $\mathbf{0.1797 \pm 0.0077}$ \\
\midrule
subset‑Magnitude       & 500 & $0.3750 \pm 0.0183$ & $0.0455 \pm 0.0017$ & $0.1795 \pm 0.0065$ \\
lowrank‑Diagonal       & 500 & $0.3720 \pm 0.0178$ & $0.0453 \pm 0.0038$ & $0.1793 \pm 0.0065$ \\
lowrank‑KFAC           & 500 & $0.3718 \pm 0.0178$ & $0.0454 \pm 0.0037$ & $0.1792 \pm 0.0065$ \\
lowrank-KFAC + $1/\lambda$ shift & 500 & $0.3595 \pm 0.0153$ & $0.0333 \pm 0.0031$ & $0.1803 \pm 0.0076$ \\
\textbf{$\Sigma_m$ (ours)} & \textbf{500} & $\mathbf{0.3573 \pm 0.0156}$ & $\mathbf{0.0317 \pm 0.0048}$ & $\mathbf{0.1792 \pm 0.0077}$ \\
\midrule
\multicolumn{5}{c}{\textit{Full‑rank baselines ($s = p \approx 672{,}890$) }} \\
\midrule
Diagonal LA            & $p$ & $0.3684 \pm 0.0143$ & $0.0428 \pm 0.0022$ & $0.1836 \pm 0.0072$ \\
KFAC LA                & $p$ & $0.3715 \pm 0.0141$ & $0.0471 \pm 0.0030$ & $0.1846 \pm 0.0071$ \\
ELLA (Nyström)         & $p$ & $0.3600 \pm 0.0153$ & $0.0347 \pm 0.0030$ & $0.1806 \pm 0.0076$ \\
MAP                    & --   & $0.3762 \pm 0.0185$ & $0.0434 \pm 0.0015$ & $0.1796 \pm 0.0065$ \\
\bottomrule
\end{tabular}%
}
\end{table}

The predictive metrics separate the methods far less than the KL does. On FashionMNIST at $s=1000$, $\Sigma_m$ reaches the exact full Laplace (NLL $0.2499$ against $0.2498$, Brier $0.1313$ against $0.1312$), which is the most that can be asked of it, and no other approximation gets that close on NLL. The absolute differences are nevertheless small: every method sits within a few thousandths of the MAP, and the ordering on ECE is not the ordering on KL. The diagonal Laplace has the best ECE ($0.0152$) while its KL is $5.25$, and $\Sigma_m$ has ECE $0.0168$ with KL $0.016$. We report both tables because a
method can be well calibrated and still be a poor posterior, and only the KL table measures the latter. On CIFAR-10 the same pattern holds: the differences in NLL, ECE and Brier are within one or two standard errors for almost every method, while the KL values differ by four orders of magnitude.

\subsubsection{Fidelity: KL divergence and log-trace}

The Figures~\ref{fig:fashionmnist_cifar10_scaling} and the table~\ref{tab:fullrank_kl_logtr}  show the per-point Gaussian KL divergence to the exact predictive covariance, the log-trace of the predictive covariance for  the  range of subspace dimensions s and  Full-rank Laplace baselines (diagonal, KFAC, ELLA, and the exact full Laplace) respectively for the two models used here. The exact full Laplace is still  used as the reference.

\paragraph{Subspace Scaling fidality} Figures~\ref{fig:fashionmnist_cifar10_scaling} show how the KL divergence and $\log \mathrm{Tr}(\Sigma_X)$ evolve as the subspace dimension $s$ grows.  $\Sigma_m$ quickly approaches the exact
posterior: on FashionMNIST it already surpasses full‑rank KFAC ($\mathrm{KL}=0.9755$) at $s = 230$ ($\mathrm{KL}=0.5521$), and on CIFAR‑10 it achieves $\mathrm{KL}=0.0087$ already at $s=200$.  Because $\Sigma_m$ is not a subspace projection, it is not bounded by the ceiling of Lemma~2 \citep{faller2025low} and can approach the exact trace from above.

\begin{figure}[!htbp] 
    \centering
    \includegraphics[width=\linewidth]{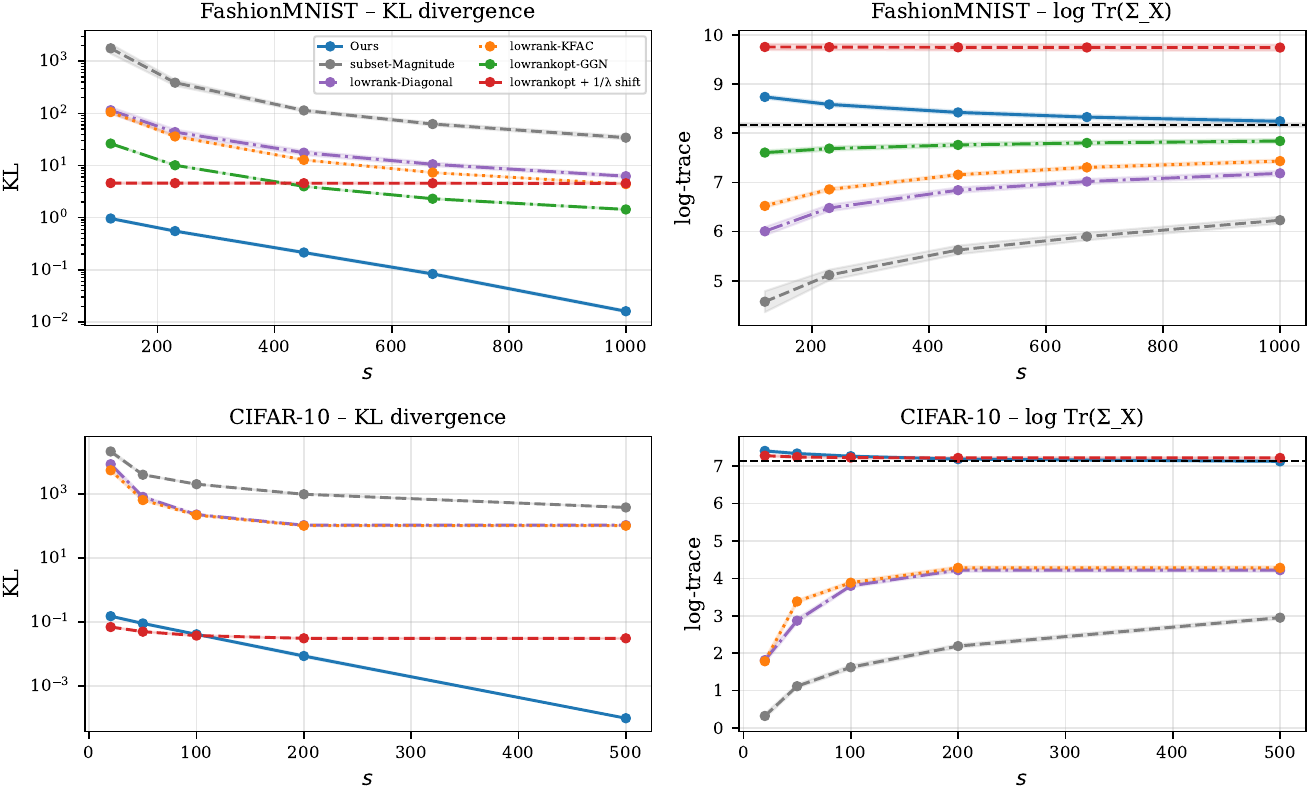}
    \caption{\textbf{FashionMNIST CNN subspace scaling.} $\Sigma_m$ reaches
    near‑exact fidelity at $s=1000$ ($\mathrm{KL}=0.0160$) and surpasses full‑rank
    KFAC already at $s=230$. \\ \textbf{CIFAR‑10 ResNet‑9 subspace scaling.} The corrected
    covariance scales smoothly towards the exact posterior, outperforming
     all subspace‑projection baselines at every $s$}
    \label{fig:fashionmnist_cifar10_scaling}
\end{figure}

On CIFAR-10 the shift baseline is more accurate than $\Sigma_m$ at the three smallest subspace dimensions: its KL is $0.0698$, $0.0501$ and $0.0377$ at $s=20$, $50$ and $100$, against $0.1520$, $0.0903$ and $0.0413$ for $\Sigma_m$. $\Sigma_m$ overtakes it at $s=200$ ($0.0087$ against $0.0310$). The reason is structural. The GGN is assembled from a subset of the training data, so its rank is far below $p$ and $A^{-1}$ is already low rank plus a multiple of the identity, which is the form the shift baseline assumes. At small $s$ the Krylov
basis has not resolved enough of the low-rank part to compensate. On
FashionMNIST, where the GGN is not rank-deficient in this way, $\Sigma_m$ is more accurate than the shift at every $s$.

\paragraph{Full‑rank fidelity} The KL divergence and the log‑trace of the predictive covariance are invariant to the subspace dimension $s$ for the full‑rank baselines.
Table~\ref{tab:fullrank_kl_logtr} reports these fidelity metrics for both FashionMNIST and CIFAR‑10 using the mean over 5 seeds.  The exact
full Laplace (reference) has $\mathrm{KL}=0$ and its log‑trace
$\log\mathrm{Tr}(\Sigma_X)$ is the ceiling that any rank‑$s$ projection must obey (Lemma~2 of \citet{faller2025low}).  For comparison we also include our corrected covariance $\Sigma_m$ at the largest subspace dimension used in the subspace sweep ($s=1000$ on FashionMNIST, $s=500$ on CIFAR‑10).

\begin{table}[H]
\centering
\caption{KL divergence and log-trace of full-rank baselines (mean over 5 seeds). The CG reference is accurate to a relative residual of about $10^{-6}$, so KL values below $10^{-4}$ are at the resolution of the reference itself and are reported as an upper bound.}
\label{tab:fullrank_kl_logtr}
\begin{tabular}{lcc@{\hskip 10pt}c@{\hskip 10pt}c@{\hskip 10pt}c@{\hskip 10pt}c}
\toprule
\multicolumn{1}{c}{\textbf{Dataset}} & \multicolumn{1}{c}{\textbf{metric}} &
\multicolumn{1}{c}{\textbf{Exact Full LA}} &
\multicolumn{1}{c}{\textbf{Diagonal LA}} &
\multicolumn{1}{c}{\textbf{KFAC LA}} &
\multicolumn{1}{c}{\textbf{ELLA}} &
\multicolumn{1}{c}{\textbf{$\Sigma_m$ (ours)}} \\
\midrule
\multirow{2}{*}{FashionMNIST} & KL                          & 0.0000 & 5.2486 & 0.9755 & 7.8998 & \textbf{0.0160} \\
                              & $\log\mathrm{Tr}(\Sigma_X)$  & 8.1751 & 7.9667 & 8.5422 & 7.2130 & \textbf{8.2429} \\
\midrule
\multirow{2}{*}{CIFAR‑10}     & KL                          & 0.0000 & 0.3609 & 0.3907 & 0.3767 &  \textbf{ $< 10^{-4}$  } \\
                              & $\log\mathrm{Tr}(\Sigma_X)$ & 7.1270 & 7.5088 & 7.5902 & 7.1994 & \textbf{7.1297} \\
\bottomrule
\end{tabular}
\end{table}

Our method reduces the KL divergence by more than an order of magnitude
compared to the best full‑rank baseline (KFAC) on both datasets, and its
log‑trace sits at the exact ceiling on CIFAR‑10, confirming that the
trace‑matched isotropic bulk correctly accounts for the complement variance.
On FashionMNIST the trace is slightly above the ceiling, consistent with the
fact that $\Sigma_m$ is not bound by the subspace upper bound of
Lemma~2.

\subsubsection{Key Observations}
\begin{enumerate}
   
    \item \textbf{Fidelity against the baselines.}
    On FashionMNIST at $s=1000$, $\Sigma_m$ reaches $\mathrm{KL}=0.0160$. The closest competitor is full-rank KFAC at $0.9755$, then the Theorem~1 optimum at $1.44$, the low-rank-plus-shift at $4.55$, the diagonal at $5.25$ and ELLA at $7.90$. On CIFAR-10 at $s=200$ we get $0.0087$ against $0.3907$ for full-rank KFAC.

    \item \textbf{Where the shift baseline is still better.}
    On CIFAR-10 the low-rank-plus-shift is more accurate than $\Sigma_m$ at $s=20$, $50$ and $100$, and $\Sigma_m$ only overtakes it from $s=200$. The GGN there is built from a subset of the training data, so its rank is far below $p$ and the inverse is already low rank plus a multiple of the identity. That is the form the shift baseline assumes, so it is well specified in this case. On FashionMNIST, where the GGN is not rank-deficient in the same way, $\Sigma_m$ is ahead at every $s$.

    \item \textbf{The subspace ceiling does not bind $\Sigma_m$.}
    Lemma 2 of \citet{faller2025low} bounds the trace of any rank-$s$
    projection from above by $\log\mathrm{Tr}(\Sigma_X)$. $\Sigma_m$ is full rank, so the bound does not apply to it. On CIFAR-10 it sits at the ceiling ($7.1297$ against the exact $7.1270$) and on FashionMNIST slightly above it ($8.2429$ against $8.1751$). Every projection baseline sits below the ceiling on both datasets, which is the under-coverage the bound predicts.

    \item \textbf{Two regimes, and a way to tell them apart in advance.}
    On CIFAR-10 we measure $\lambda\omega \approx 0.999$, so the isotropic bulk has essentially become the $1/\lambda$ prior shift and the complement is dominated by the prior. What separates $\Sigma_m$ from the shift baseline there is the exact boundary coupling, not $\omega$. On FashionMNIST $\lambda\omega$ runs from $0.933$ to $0.981$ and the bulk carries curvature of its own. Both $\lambda\omega$ and the cross term $\|c\|$ of Proposition~\ref{prop:total_quality}(iii) come at no extra matrix-vector cost, so a user can check which regime they are in before trusting the correction.

    \item \textbf{Degenerate baseline.}
    The parameter-subset method \texttt{subset-Diagonal} selects blocks with numerically zero variance, so its KL is infinite. \citet{faller2025low} report the same behaviour and attribute it to the large fraction of parameters with  gradient almost equal to zero. We exclude it from the tables.
\end{enumerate}

\section{Conclusion}
\label{sec:conclusion}

Truncating the Lanczos approximation of $A^{-1}$ discards the orthogonal complement of the Krylov subspace and when $A$ is ill-conditioned, it loses the most information because the orthogonal complement holds the largest contributions to the inverse. We have shown that this ignored residue retains more information than an error term because, in fact  expressed in the Krylov basis, the resolved and unsolved subspaces are coupled only through the single residue $\beta_m$, so the true complement block of $A^{-1}$ is the Schur complement $S_{\perp}^{-1}$ rather than the compressed operator $T_{m,\perp}^{-1}$, and the standard low-rank approximation is recovered as the special case $\beta_m=0$. The corrected covariance $\Sigma_m$ retains this boundary coupling exactly, matches the total discarded variance exactly ($\tr(\Sigma_m)=\tr(A^{-1})$, Proposition~\ref{prop:total_quality}), and approximates only the coupling-free bulk, isotropically, with an error controlled by the single computable diagnostic $\kappa_S^{\perp}(m)$. $\Sigma_m$ can be used  in two ways   as a matrix-free approximation of $A^{-1}$ for sampling, and the covariance of a Gaussian belief over $A^{-1}v$ that contracts to the exact answer as computation proceeds, extending the computation-aware program of probabilistic numerics from linear solves to the matrix inverse.  P-SLQ delivers the required complement trace from a single reused Krylov basis, with exponentially decaying bias, variance at the Monte Carlo rate, and a sample complexity guarantee (Theorem~\ref{thm:pslq_consistency}, Corollary~\ref{cor:pslq_rate}); a single boundary probe and a deterministic Sherman-Morrison step then yield the coupling-corrected magnitudes.

We tested our P-SLQ on the estimation of $\operatorname{tr}(A^{-1})$ against Hutchinson+CG, SLQ, Hutch++, XTrace and the unbiased solver ULISSE. At equal matrix-vector budgets P-SLQ gave the lowest relative RMSE on all four matrices we used, including the kernel matrix of figure~\ref{fig:exp_ulisse}, which we did not construct from a chosen spectrum. This is the trade-off we wanted: a bias that decays exponentially with the quadrature depth, in exchange for a variance of order $O(N^{-1})$ that does not depend on $m$.  For the Laplace approximation we compared $\Sigma_m$ to the exact posterior. On FashionMNIST the KL divergence is $0.016$ at $s=1000$. The best baseline is KFAC at $0.98$, then the low-rank-plus-shift at $4.55$, the diagonal at $5.25$, ELLA at $7.90$ and the Theorem~1 optimum at $1.44$. On CIFAR-10 the picture is not the same, and we think the reason is structural. The GGN there is built from a subset of the training data, so its rank is much smaller than $p$ and the inverse is already low rank plus a multiple of the identity. That is the form the shift baseline assumes, and it is well specified in this case. It is more accurate than $\Sigma_m$ at $s=20$, $50$ and $100$; $\Sigma_m$ overtakes it at $s=200$ and reaches a KL below $10^{-4}$ at $s=500$. We report the crossover because it says where the method does and does not help. On FashionMNIST, where the GGN is not rank-deficient in this way, $\Sigma_m$ is more accurate than the shift at every $s$.  The predictive metrics separate the methods much less than the KL does. At $s=1000$ on FashionMNIST $\Sigma_m$ reaches the exact Laplace posterior (NLL $0.2499$ against $0.2498$, Brier $0.1313$ against $0.1312$), which is the most that can be asked of it, but the other approximations are also within a few thousandths. We report these numbers next to the KL table and not instead of it. A method can be well calibrated and still be a poor posterior.

The isotropic bulk is a summary of the complement, and it is only as good as the spread of the complement spectrum. When the spectrum decays quickly, $\kappa_S^{\perp}(m)$ stays large, one variance cannot represent the complement, and the approximation gets worse. Theorem~\ref{thm:iso} says this will happen and the experiments of Section~\ref{sec:exp_pslq} show it happening.


\vskip 0.2in
\bibliography{sample}

\end{document}